\documentclass[10pt]{amsart}
\usepackage{bbm}
\usepackage{mathrsfs}
\usepackage[shortlabels]{enumitem}
\usepackage[usenames,dvipsnames,svgnames,table]{xcolor}
\usepackage[pagebackref,colorlinks=true,linkcolor=BrickRed,citecolor=OliveGreen,pdfstartview=FitH]{hyperref}
\usepackage{graphicx}
\usepackage{float}
\usepackage[normalem]{ulem}
\usepackage{comment}
\usepackage{enumitem}

\newtheorem{thm}{Theorem}[section]
\newtheorem{ass}[thm]{Assumption}
\newtheorem{coro}[thm]{Corollary}
\newtheorem{lem}[thm]{Lemma}
\newtheorem{prop}[thm]{Proposition}

\theoremstyle{definition}

\newtheorem{defn}[thm]{Definition}

\theoremstyle{remark}

\newtheorem{remk}[thm]{Remark}

\renewcommand{\th}{\theta}

\renewcommand{\index}{{\rm ind}}
\newcommand{\ind}{{\rm ind}}

\newcommand{\oD}{\overline{D}}
\newcommand{\oz}{\overline{z}}

\newcommand{\oP}{\overline{P}}

\newcommand{\crit}{{\rm crit}}

\newcommand{\tu}{\widetilde{u}}

\newcommand{\tp}{\widetilde{p}}

\newcommand{\tilh}{\widetilde{h}}

\newcommand{\Rbb}{ {\mathbb R}}
\newcommand{\Zbb}{ {\mathbb Z}}

\newcommand{\Cbb}{ {\mathbb C}}

\newcommand{\Zcal}{ {\mathcal Z}}
\newcommand{\Scal}{ {\mathcal S}}

\newcommand{\Ncal}{ {\mathcal N}}

\renewcommand{\phi}{\varphi}

\definecolor{blue(ncs)}{rgb}{0.0, 0.53, 0.74}

\title{Critical points of Laplace eigenfunctions on polygons}

\author{Chris Judge}
\address{Department of Mathematics, Indiana University, Bloomington} 
\email{\href{mailto:cjudge@indiana.edu}{cjudge@indiana.edu}}
\thanks{The work of C.J.\ is partially supported by a Simons Foundation collaboration grant. }

\author{Sugata Mondal}
\address{School of Mathematics, Tata Institute of Fundamental Research, Mumbai}
\email{\href{mailto:sugatam@math.tifr.res.in}{sugatam@math.tifr.res.in}}
\thanks{The work of S. M. \ is partially supported by Ramanujan Fellowship of SERB, Govt. of India.}

\begin{document}

\begin{abstract}
We study the critical points of Laplace eigenfunctions on polygonal 
domains with a focus on the second Neumann eigenfunction. 
We show that if each convex quadrilaterals has no second Neumann 
eigenfunction with an interior critical point,
then there exists a convex quadrilateral with an unstable critical point.
We also show that each critical point of a second-Neumann eigenfunction
on a Lip-1 polygon with no orthogonal sides is an acute vertex.
\end{abstract}

\maketitle

\section{Introduction}

A second Neumann eigenfunction $u$ of the Laplacian approximates 
the temperature distribution of an insulated domain for large times.  
The `hot spots' conjecture \cite{Rauch} \cite{Kawohl} is the assertion
that $u$ does not assume its maximum value in the interior of the domain.
The conjecture is false for some non-contractible plane domains
\cite{B-W} \cite{Burdzy-one-hole} but is still believed to be true 
for convex domains. The conjecture is known to be true when 
the domains are somewhat elongated, for example, the Lip-1 planar 
domains of \cite{A-B}. In \cite{J-M} \cite{Erratum} we show that the 
hot spots conjecture holds true for acute triangles thus resolving 
Polmath 7 \cite{polymath}. 

In the present paper, we extend our study of critical points of 
eigenfunctions to general polygons and we encounter new phenomena.
Note that every planar domain may be approximated by polygonal
domains, and hence the weak form of the hot spots conjecture---{\it some 
second Neumann eigenfunction has no interior maximum}---for all planar 
domains would  follow from the verification of the strong hot spots 
conjecture---{\it every second Neumann 
eigenfunction has no interior maximum}---for all polygonal domains. 

Our general approach to the hot spots conjecture 
is based on the fact that eigenfunctions 
and their critical points\footnote{
With the exception of rectangles, the critical set 
of a second Neumann eigenfunction
on a simply connected polygon

is finite \cite{J-M-arc}.}
vary continuously as one varies the domain. Roughly speaking,
to show that a second Neumann
eigenfunction $u_0$ on a polygon $U_0$ has no interior critical points, 
one constructs a path of polygons $P_t$ and associated 
path of eigenfunctions $u_t$ so that the eigenfunction $u_1$ on $P_1$ has no 
interior critical points. If one can show that the putative critical 
points of each $u_t$ are `stable' under perturbation, then $u_0$ 
also has no interior critical points. 

In the case of triangles, we took $P_1$ to be a right isosceles triangle,
and we established enough stability to successfully implement 
this strategy \cite{J-M} \cite{Erratum}.
Here we show that the strategy is likely to be more 
difficult to implement if the polygon has more sides. 

\begin{thm}
\label{thm:unstable}
If each convex quadrilateral has no interior critical point, 
then there exists a convex quadrilateral $Q$, a second Neumann 
eigenfunction $u$ on $Q$, and a nonvertex critical point $p$ of $u$
that is not stable under perturbation.
\end{thm}

By `stable under perturbation' we mean that if $Q_n$ is a sequence
of quadrilaterals that converges to $Q$ and $u_n$ is a sequence 
of second Neumann eigenfunctions on $Q_n$ that converges to $u$,
then each $u_n$ has a critical point $p_n$ so that $p_n$ converges to $p$.
We conjecture that instability does not hold for triangles.

On the other hand, we are able to successfully apply our strategy for 
resolving the hot spots conjecture to a large class of polygons. 

\begin{thm}
\label{thm:admissible}
Suppose that $P_t$ is a path of polygons such that each $P_t$ has exactly two
acute vertices, no two sides of $P_t$ are orthogonal, and $P_1$ is
an obtuse triangle. Then the second Neumann eigenvalue of $P_0$ is simple, 
and the set of critical points of each eigenfunction consists
of the two acute vertices.
\end{thm}

The class of polygons described in Theorem \ref{thm:admissible}
is exactly (up to rigid motion) the class 
of polygons that have no orthogonal sides and satisfy the Lip-1
condition of \cite{A-B}  (see Proposition \ref{prop:lip-admissible}).
Thus, Theorem \ref{thm:admissible} provides a non-probabilistic proof
of the weak hot spots conjecture for Lip-1 domains.  
Moreover, in contrast to the result of \cite{A-B}, 
we find that not only are there no interior critical points but
there are also no critical points on the boundary other than 
the two acute vertices. 
Recently, Jonathan Rohleder \cite{Roh} announced a non-probabilistic
proof of the main result of \cite{A-B}.

We now outline the contents of this paper.
In section \ref{sec:sector}, we use the Bessel expansion of
an eigenfunction $u$ to understand the nodal set of $Xu$ near 
a vertex where $X$ is a constant (resp. rotational) vector field. 
In particular, we show that whether or not an arc of
the nodal set $Xu$ ends at the vertex is essentially determined
by the first two Bessel coefficients, the angle at $v$,
and the angle between the vector field and the sides 
adjacent to $v$ (resp. location of central point).
These criteria  will be used crucially in the 
proof of Theorem \ref{thm:admissible}.

We will need to rule out the possibility that critical points 
of a sequence of eigenfunctions, associated to a convergent 
sequence of polygons, converge to a vertex of the limiting polygon. 
In \S \ref{sec:crit-pts-converge-vertex} we show in various 
contexts that if critical points converge to a vertex $v$, 
then the first two Bessel coefficients of the limiting
eigenfunction equal zero. If the limiting polygon is simply
connected then this is impossible (Proposition \ref{two:coeff:zero}).

To check the stability of a critical point under perturbation, 
we will use a variant of the Poincar\'e-Hopf index.
In \S \ref{sec:topology}, we define this invariant to
include vertices and we prove a variant of the Poincar\'e-Hopf
index formula for Neumann eigenfunctions $u$ on polygons. 
We relate the index of a critical point of $u$ located at a vertex
$v$ with the first two Bessel coefficients of $u$ at $v$.
We also show that the `total local index' is unchanged under 
perturbation (Theorem \ref{prop:local-number-of-critical-points}).
As a consequence each non-zero index critical point is stable
(Lemma \ref{lem:top-stable-implies-stable}).

In \S \ref{sec:zero-index} we provide a local normal form for an 
eigenfunction in a neighborhood of a critical point $p$ of $u$ whose 
Poincar\'e-Hopf index equals zero (Lemma \ref{lem:tangential-cusp}).
Using this local normal form, we find that an index zero critical 
point cannot be a degree 1 vertex of the nodal set of $Xu$ where 
$X$ is either a constant or rotational vector field.

In \S \ref{sec:simply-connected} we specialize to simply connected polygons. 
For such domains, the nodal set of a second Neumann eigenfunction $u$
is a simple arc, and from this fact we deduce that at least 
one of the first two Bessel coefficients at each vertex is nonzero.
This implies a tighter relationship between the index of a vertex
critical point of $u$ and the first two Bessel coefficients
(Corollary \ref{coro:1-connected-index}).

In \S \ref{sec:eigen-obtuse} we prove Theorem \ref{thm:admissible}
(Theorem \ref{thm:only-acute}).
We first show that if polygon $P$ has at least one acute vertex
and a second Neumann eigenfunction $u$ on $P$ has an interior
critical point, then either $u$ has four non-zero index critical 
points or there exists a side $e$ of $P$ such that the nodal 
set of the derivative of $u$ in the direction $u$ 
has an arc that ends at a vertex $v$ of $P$.
This leads us to consider, for the path $u_t$ in Theorem 
\ref{thm:admissible}, the number, $S(t)$, of nonzero
index critical points and the number, $V(t)$, of vertices 
that are endpoints of an nodal arc of the derivative of $u_t$
in the direction of a side of $P$. We show that the set $A$
of $t \in [0,1]$ such that either $S(t)\geq 3$ or $V(t) \geq 1$ 
is open and closed. For the obtuse triangle $P_1$,
we have $S(1)=0$ and $V(1)=0$, and hence $A$ is empty.
In particular, the initial polygon $P_0$ has at most two non-zero
index critical points, and from this we deduce 
using the results of \S \ref{sec:zero-index} that there are
zero index critical points. Using the fact that $V(0)=0$, we find
that the two critical points are located at the vertices of $P_0$.
These two critical points are the unique 
global extrema, and this implies
that the eigenspace is one-dimensional. 

In \S \ref{sec:blocking} we provide a criterion for the instability 
of a critical point on a quadrilateral. This criterion is based on 
the fact that the index of a vertex with angle less than $\pi$ 
cannot equal $-1$ (Corollary \ref{coro:1-connected-index}). 
In particular, an index $-1$ critical point
cannot cross from one side adjacent to a vertex to the other side
of the vertex if the angle at the vertex is in $(\pi/ 2, \pi)$. 
Hence one is led to find a path of quadrilaterals $Q_t$
such that $Q_0$ has an index $-1$ that lies on one side
of a vertex and $Q_1$ and has an index $-1$ critical point 
on the other side of the vertex.

In \S \ref{sec:break} we construct such a path af quadrilaterals 
and thus prove Theorem \ref{thm:unstable} (Theorem \ref{thm:unstable-break}). 
The path is constructed by
taking a nearly isosceles triangle whose vertex $v$ of smallest
angle is less than $\pi/3$, and then `breaking' the side opposite to $v$.

In \S \ref{sec:convex} we specialize to convex polygons 
and find that if a second Neumann eigenfunction has only three
critical points then one is a minimum, one is a maximum, and the
third has index zero.

\section{Eigenfunctions on a sector} \label{sec:sector}

To understand the behavior of an eigenfunction in a neighborhood a vertex $v$ 
of angle $\beta$ of a polygon, we will consider its Fourier-Bessel expansion.
By performing a rigid motion, we may assume that the vertex $v$ 
is the origin, one side adjacent to $v$ lies in the ray 
$\{ z= r\, :\, r \geq 0\}$
the nonnegative real axis, and the other
side lies in the ray $\{ z= r \cdot e^{i \beta}\, :\, r \geq 0\}$.
If $u$ is a (real) eigenfunction with eigenvalue $\mu$ that satisfies 
Neumann conditions on the the rays $\theta=0$ and $\theta = \beta$, 
then separation of variables leads to the Fourier-Bessel expansion:
\begin{equation}  \label{eqn:Bessel-expansion}
u \left(r e^{i\theta} \right)~ =~ \sum_{n=0}^\infty\, c_n \cdot J_{\frac{n\pi}{\beta}}(\sqrt{\mu} \cdot r) \cdot
 \cos \left({\frac{n\pi \theta}{\beta}}\right).
\end{equation}
Here $c_n \in \Rbb$ and $J_{\nu}$ denotes the Bessel function of 
the first kind of order $\nu$ \cite{Lebedev}
\begin{equation} \label{eqn:exp-of-Bessel}   
J_{\nu}(x)~ 
=~ 
x^{\nu} \cdot \sum_{k=0}^{\infty}\, 
 \frac{(-1)^k \cdot x^{2k}}{2^{2k} \cdot \Gamma(k+\nu) \cdot \Gamma(k+\nu +1) }
\end{equation}
where $\Gamma$ is the Gamma function. 

If $u$ satisfies Dirichlet conditions on the rays $\theta= 0$ and $\theta=\beta$, 
then the one replaces $\cos$ with $\sin$, 
and if $u$ satisfies Dirichlet conditions
on the ray
$\theta=0$ and Neumann conditions on the ray
$\theta=\beta$,
then one replaces $\cos(n \pi \theta/\beta)$ with $\sin(n \pi \theta/2\beta)$
and $J_{n \pi/\beta}$ with  $J_{n \pi/2\beta}$.

From (\ref{eqn:exp-of-Bessel}) we find that, for each $\nu \geq 0$, 
there exists an entire function $g_{\nu}$ so that 
$J_{\nu}(\sqrt{\mu} \cdot r) = r^{\nu} \cdot g_{\nu}(r^2)$.\footnote{Note 
that though $g_{\nu}$ depends on the eigenvalue $\mu$, we will suppress 
$\mu$ from the notation.} Note that neither $g_{\nu}$ nor $g_{\nu}'$  vanishes 
in a neighborhood of $0$ for each $\nu\geq 0$. 
With this notation, (\ref{eqn:Bessel-expansion})
takes a more compact form
\begin{equation}  \label{eqn:Bessel-expansion-II}
u \left(r e^{i\theta} \right)~ 
=~ 
\sum_{n=0}^\infty\, c_n \cdot r^{n \cdot \nu} \cdot g_{n \cdot\nu}\left(r^2\right) \cdot
 \cos \left(n \cdot \nu \cdot \theta \right)
\end{equation}
where $\nu = \pi /\beta$.  Note that we are suppressing the dependence of $g$ on 
the eigenvalue $\mu$.

Given a function $f$, let $\Zcal(f)= f^{-1}(0)$ denote
the {\em nodal set} of $f$. 
For each $\psi \in \Rbb$, 
let $L_{\psi}$ denote the constant vector field defined by
\begin{equation}  \label{eqn:constant-field}
L_{\psi} u~ =~ \cos(\psi) \cdot \partial_x~ +~ \sin(\psi) \cdot \partial_y.
\end{equation}

\begin{lem} \label{lem:nodal-sector}
Let $u$ be an eigenfunction that satisfies Neumann conditions 
on the rays $\theta=0$ and $\theta=\beta$. 
\vspace{.2cm}
\begin{enumerate}[label=(\alph*)]
\item \label{part-a}
If $c_0 \neq 0$ and either $0<\beta<\pi/2$ or $c_1=0$,
then there exists an arc in $\Zcal(L_{\psi} u)$  
with an endpoint at the vertex $v$
if and only if $\psi \in [\pi/2,\pi/2 + \beta] \mod \pi$.

\vspace{.2cm}

\item  \label{part-b}
If $c_1 \neq 0$ and $\pi/2<\beta<\pi$,
then there exists an arc in $\Zcal(L_{\psi} u)$ 
with an endpoint at the vertex
if and only if $\psi \in [\beta-\pi/2, \pi/2] \mod \pi$.

\item  \label{part-c}

If $c_1 \neq 0$ and  $\beta =\pi$,
then there exists an arc in $\Zcal(L_{\psi} u)$ 
with an endpoint at the vertex
if and only if $\psi =\pi/2  \mod \pi$. 
Moreover, near the vertex $v$, this arc lies on $\partial P$.

\item  \label{part-d}

If $c_1 \neq 0$ and  $\beta >\pi$,
then there exists an arc in $\Zcal(L_{\psi} u)$ 
with an endpoint at the vertex
if and only if $\psi  \in [\pi/2, \beta - \pi/2]  \mod \pi$.

\end{enumerate}

Moreover, in all of the above situations, $\Zcal(L_{\psi} u)$ 
has at most one arc with an endpoint at the origin.
\end{lem}

\begin{figure}
\includegraphics[scale=.3]{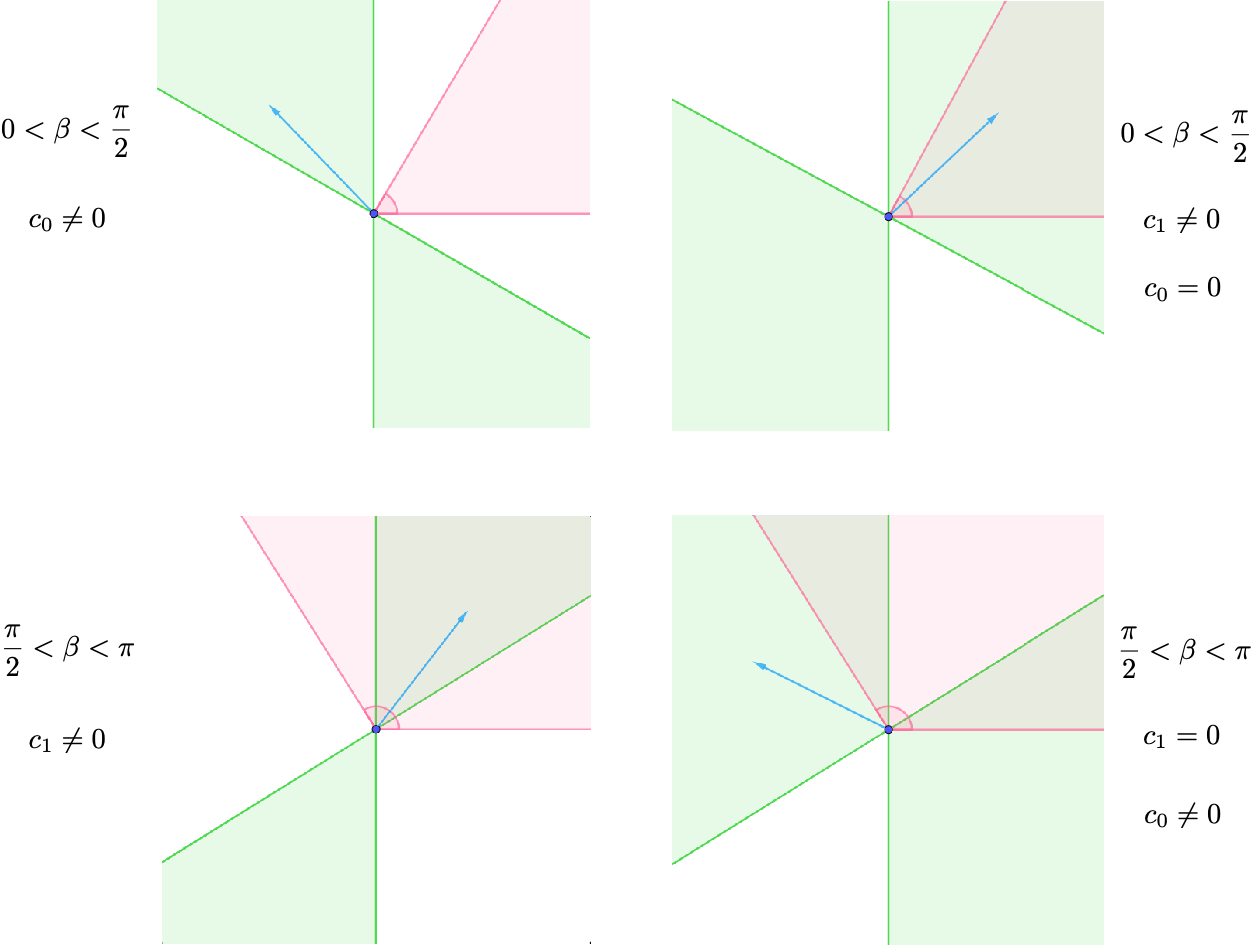}
\caption{Some of the cases described by Lemma \ref{lem:nodal-sector}. 
The region shaded pink is a neighborhood of the vertex $v$ with angle $\beta$.
The region shaded green describes the directions $L_{\psi}$ 
for which $L_{\psi} u$ has a nodal arc ending at $v$.
\label{fig:4cases}}
\end{figure}

\begin{proof}
If $\psi = \pi/2 \mod \pi$ (resp. $\psi= \beta -\pi/2 \mod \pi$) then
because $u$ satisfies Neumann conditions, the arc corresponding to 
$\theta=0$ (resp. $\theta =\beta$) lies in both $\Zcal(L_{\psi} u)$.

Because $\partial_x = \cos(\theta) \cdot \partial_r - \sin(\th) r^{-1} \cdot \partial_{\th}$ 
and $\partial_y = \sin(\theta) \cdot \partial_r + \cos(\th) r^{-1} \cdot \partial_{\th}$,
we have 
\begin{equation} \label{eqn:Lpsi-polar}
L_{\psi}~
=~ \cos(\psi-\theta) \cdot \partial_r~
+~
\sin(\psi-\theta) \cdot \frac{1}{r} \partial_{\theta}.
\end{equation}
If $c_0 \neq 0$ and either $c_1=0$ or $0 < \beta < \pi/2$, 
inspection of (\ref{eqn:Bessel-expansion-II}) shows that 
\begin{equation} \label{eqn:c_0-dominant}
u \left( r e^{i \theta} \right)~ 
=~  
a~ +~ b \cdot r^{2}~ +~ f(r, \theta)
\end{equation}
where $a$ and $b$ are constants and $|L_{\psi} f|= o(r)$ 
and $|\partial_{\theta} L_{\psi} f|=o(r)$.  In particular, 
we find that
\begin{equation} \label{eqn:c_1-L-u}
L_{\psi} u\left( r e^{i \theta} \right)~
=~
2 b \cdot r \cdot \cos(\psi-\theta)~
+~
o(r),
\end{equation}
and  
\begin{equation} \label{eqn:c_1-theta-L-u}
\partial_{\theta} L_{\psi} u \left( r  e^{i \theta} \right)~
=~
- 2 b \cdot r \cdot \sin(\psi-\theta)~
+~
o\left(r \right).
\end{equation}

If $\psi \in (\pi/2, \pi/2+ \beta) \mod \pi$, 
then from (\ref{eqn:c_1-L-u}) we find that $L_{\psi}u(r)$ and 
$L_{\psi}u(r e^{i \beta})$ have opposite signs  for sufficiently small $r$. 
Thus, by the intermediate value theorem,
there exists $\gamma(r) \in (0,\beta)$ so that 
$(L_{\psi}u) \left(r e^{i \gamma(r)} \right)=0$. 
Moreover, from (\ref{eqn:c_1-theta-L-u}), we find that 
the point $\gamma(r)$ is unique for sufficiently small $r$, 
and by the implicit function theorem, $\gamma$ is smooth. 
The map $r \mapsto r e^{i \gamma(r)}$ is the desired arc in $\Zcal(L_{\psi} u)$.
The uniqueness of $\gamma(r)$ implies that there is 
at most one arc.

Thus we have proven \ref{part-a}.
To prove \ref{part-b}, note that if  $\beta > \pi/2$
or $c_0=0$,
then inspection of (\ref{eqn:Bessel-expansion-II})
shows that
\[
u\left(r\cdot e^{i \theta}\right)~
=~
c_0~ 
+~ 
c_1 \cdot r^{\nu} \cdot \cos(\nu \cdot \theta)~
+~
o(r^{\nu}).
\]
and hence a straightforward computation gives
\[ 
L_{\psi} u\left(r\cdot e^{i \theta}\right)~
=~ 
c_1 \cdot \nu \cdot r^{\nu-1} 
\cdot 
\cos\left(\psi + (\nu-1) \cdot \theta \right)~
+~
o(r^{\nu-1}).
\]
Note that 
$\cos\left(\psi + (\nu-1) \cdot \theta \right)$
vanishes if and only if 
\[ 
\psi~
=~
\frac{\pi}{2} + (\beta - \pi) \cdot \frac{\theta}{\beta}
\mod \pi.
\]
One argues as in the proof of part (a) to obtain a unique
arc of $\Zcal(L_{\psi}u)$ that ends at $v$.

Parts \ref{part-c} and \ref{part-d}
follow from arguments similar those
that verified \ref{part-a} and \ref{part-b}.

\end{proof}

Lemma \ref{lem:nodal-sector} motivates the following definition.

\begin{defn}
Suppose $\beta \neq \pi/2$ and $u$ is a Neumann eigenfunction 
on the sector of angle $\beta$. 
We define the {\em leading Bessel coefficient} 
of $u$ at $v$ to be 
\begin{itemize}
\item $c_0$  if $\beta < \pi/2$, and 

\item $c_1$ if $\beta > \pi/2$.
\end{itemize}
\end{defn}

The definition allows us to succinctly state the following corollaries.

\begin{coro} \label{coro:arc-iff-leading-vanish}
If $L$ is a nonzero
constant vector field parallel to one the boundary rays of the sector 
of angle $\beta \neq \pi/2$, then the leading coefficient vanishes 
if and only if an arc of $\Zcal(L u)$ ends at $v$.
\end{coro}

\begin{proof}
Since $L$ is parallel to one of the boundary rays, 
the vector field $L$ is a multiple of $L_{\psi}$ with 
$\psi=0$ or $\beta$. In particular, $\Zcal(Lu)= \Zcal(L_{\psi}u)$.
By using the reflection symmetry about $\theta= \beta/2$,
we may assume without loss of generality that $\psi=0$.

Suppose that the leading coefficient does not vanish. 
If $\beta < \pi/2$, then  part (a) of Lemma \ref{lem:nodal-sector} 
implies that no arc in $\Zcal(L_{\psi} u)$ ends at the vertex. 
Similarly, parts (b), (c), and (d) imply that no arc in
$\Zcal(L_{\psi}u)$ ends at the vertex in the other cases.

Conversely, suppose that the leading coefficient does equal zero. 
Let $k$ be the smallest positive integer such that $c_k \neq 0$.
If $\beta < \pi/2$, then $c_0=0$ and from (\ref{eqn:Bessel-expansion-II})
we find that 
\[
u \left(r \cdot e^{i \theta} \right)~  
=~   
r^{k \cdot \nu} \cdot \cos( k \cdot \nu \cdot \theta)~ 
+~ o\left(r^{k \nu} \right).
\]
Hence using (\ref{eqn:constant-field}) we find that
\[
L_{\psi} u \left(r \cdot e^{i \theta} \right)~  
=~   
k \cdot \nu 
\cdot 
r^{k \cdot \nu-1} \cdot \cos \left( (k \cdot \nu-1)\cdot \theta \right)
+~ o\left(r^{k \nu-1} \right).
\]
Since $\beta< \pi/2$ and $k \geq 1$, the function $\cos((k \cdot \nu-1)\cdot \theta)$  vanishes  
for some $\theta \in (0, \beta)$. An implicit function 
theorem argument establishes the existence of a
smooth arc.

If $\beta> \pi/2$, then a similar argument applies to give the claim.

\end{proof}

\begin{coro}  \label{coro:Lu-arc-perturb}
Let $u$ be a Neumann eigenfunction on a sector of angle 
$\beta \neq \pi/2$ such that one of the first two Bessel coefficients of $u$ is non-zero.  If $L$ is a constant vector field such that
\begin{itemize}

\item some arc in $\Zcal(Lu)$ ends at the vertex, and

\item  $L$ is not orthogonal to a boundary ray of the sector,
\end{itemize}
then for each constant vector field $L'$ that is 
sufficiently close to $L$, some arc of $\Zcal(L'u)$ ends at the vertex.
\end{coro}

\begin{proof}
Since $\beta \neq \pi/2$, by Proposition 4.4 \cite{J-M}, there exists a neighborhood $N$ of the vertex $v$ of the sector that contains no critical points of $u$.

Since at least one of the zeroth and the first Bessel coefficients of $u$ 
at $v$ is non-zero, by Lemma \ref{lem:nodal-sector}, the set $\Zcal(L u)$ 
contains exactly one arc that ends at $v$. In particular, $u$ has opposite
signs on the the two rays of the sector. By continuity, for each constant
vector field $L'$ that is sufficiently close to $L$, some arc of $\Zcal(L'u)$
must end at some point $p_{L'}$ near $v$ (depending on $L'$) 
that lies on the boundary of the sector.

To finish the proof it suffices to show that $p_{L'}=v$
for $L'$ sufficiently close to $L$.
If $p_{L'}$ is not $v$, then since $L$ is not orthogonal to the sides of the
sector, so is $L'$, and hence,  $p_{L'}$ is a critical point of $u$. 
If $L'$ is sufficiently close then, by continuity, $p_{L'}$ lies in $N$.
This contradicts the first paragraph of the proof.
\end{proof}

Let $S_{\beta}$ denote the sector 
$\{z= r\cdot e^{i \theta}\, :\, \theta \in [0, \beta] \mbox{ mod } \pi \}$.
Recall that $R_w$ denotes the vector field that corresponds to 
rotation about $w \in \Cbb$.

\begin{coro} \label{coro-rotation-nodal-arc}
Let $u$ be an eigenfunction that satisfies Neumann conditions on the sides
$\theta=0$ and $\theta=\beta$.
Suppose that $c_0$ and $c_1$ are not both equal to zero.

\vspace{.1cm}

\begin{enumerate}
\item If  $\beta< \pi/2$

\vspace{.1cm}
\begin{enumerate} 

   \item and $c_0 \neq 0$, then an arc of $\Zcal(R_w)$
      ends at the vertex if and only if $w$ lies in $S_{\beta}$.

\vspace{.1cm}
      
   \item  and $c_0=0$, then an arc of $\Zcal(R_w)$
      ends at the vertex if and only if $w$ does not lie in $S_{\beta}$.

\end{enumerate}

\vspace{.1cm}

\item If  $\pi/2 < \beta < \pi$

\vspace{.1cm}

\begin{enumerate} 

   \item and $c_1 \neq 0$, then an arc of $\Zcal(R_w)$
      ends at the vertex if and only if $w$ does not lie in $S_{\beta}$. 
   
\vspace{.1cm}
   
   \item  and $c_1=0$, then an arc of $\Zcal(R_w)$
      ends at the vertex if and only if $w$ lies in $S_{\beta}$.

\end{enumerate}

\end{enumerate}
\end{coro}

\begin{proof}
If $w = \rho  \cdot  e^{i \phi}$, then a computation shows that the rotational
vector field about $w$ takes the form 
\[   
R_{w}~
=~
\partial_{\theta}~ 
+~ 
L_{\phi+ \frac{\pi}{2}}.
\]
Because $|\partial_{\th} u|= o(r^{\nu})$ and $|\partial_{\th}^2 u|= o(r^{\nu})$,
we find that (\ref{eqn:c_1-L-u}) and (\ref{eqn:c_1-theta-L-u}) still hold with 
$L_{\psi}$ replaced by $R_{w}$ and $\psi$ replaced by $\phi+ \pi/2$. 
Thus the argument given in the proof of Lemma \ref{lem:nodal-sector} applies.
\end{proof}

\begin{remk}
Similar statements hold for Dirichlet and mixed boundary conditions.
We leave the formulation of the statements to the reader.
\end{remk}

\section{Critical points on a sector converging to a vertex}
\label{sec:crit-pts-converge-vertex}

Let $S_n$ be a sequence of sectors that converges to a 
sector $S$. Let $u_n: S_n \to \Rbb$ be a sequence of Neumann eigenfunctions
that converges to a Neumann eigenfunction $u:S \to \Rbb$.
In this section we show that if certain types 
of critical points of $u_n$ converge to the vertex of $S$, 
then the two Bessel coefficients of the $u$ must vanish.
Some of these results are straightforward extensions of results 
in \cite{J-M}, but several are new.

Let $b_-$ and $b_+$ denote the distinct boundary rays of the sector $S$. 
Let $c_0$ and $c_1$ denote the respective Bessel coefficients 
of $u$ at the vertex of $S$.  Let $\beta$ denote the vertex angle of $S$,
and let $\nu=\pi/\beta$.

\begin{lem}[Compare Proposition 9.1 \cite{J-M}]
\label{lem:interior-crit-pt-convergence}
For each $n$, let 
$p_n$ be a critical point of $u_n$ that lies in the interior of $S_n$.
If $p_n$ converges to the vertex of $S$, then $c_1=0$.  If, in addition,
$\beta < \pi$, then $c_0=0$. 
\end{lem}

\begin{proof} 
Let $\beta_n$ be the angle of the sector $S_n$ and let 
$\nu_n=\pi/\beta_n$. 
By performing rigid motions if necessary, we may assume without
loss of generality that the vertex $S$ and each of $S_n$ is $0$
and that the boundary rays of $S_n$ are $\theta= 0, \beta_n$. 
Using (\ref{eqn:Bessel-expansion-II}) and the fact that $\sin(\alpha)$ divides 
$\sin(k \alpha)$ for each $k$, we find that
\[ 
\partial_{\theta} u_n\left(r \cdot e^{i\theta}\right)~
=~
\nu_n \cdot r^{\nu_n} \cdot \sin(\nu_n \cdot \theta) \cdot 
\left( c_1(n) \cdot g_{\nu_n}(r) ~
+~
O(r^{\nu_n}) \right). 
\]
Thus, since $p_n = r_n\exp(i \theta_n)$ is a critical point,
$0< \theta_n <\beta_n$, and $g_{\nu}(0) \neq 0$,  we find that 
$c_1(n) = O(r_n^{\nu_n})$.  In particular, since $u_n$ converges
to $u$, we have $c_1= \lim_{n \to \infty} c_1(n)=0$.

From (\ref{eqn:Bessel-expansion-II}), we find 
\[ 
\partial_r u_n\left( r \cdot e^{i n \theta}\right)~
=~
c_0(n) \cdot 2r \cdot g_0'(r^2) ~
+~
c_1(n) \cdot \nu_n \cdot r^{\nu_n-1} 
\cdot g_{\nu_n}(r^2) \cdot\cos(\nu_n \theta)~
+~
O \left(r^{\nu_n} \right).
\]
Thus, since $p_n = r_n\exp(i \theta_n)$ is a critical point, $g_0'(0) \neq0$,
and $c_1(n)= O(r^{\nu_n})$, we find that 
$c_0(n) = O(r^{2 (\nu_n -1)}) + O(r^{\nu_n-1})$.
If $\beta<\pi$, then there exists $\epsilon>0$ 
so that for sufficiently large $n$, we have $\nu_n>1+\epsilon$.
Hence $c_0 = \lim_{n \rightarrow \infty} c_0(n)=0$.
\end{proof}

\begin{lem}[Compare Lemma 9.2 \cite{J-M}]
\label{lem:simult-conv-crit-pts}
For each $n$, let $p_n$ be a critical point of $u_n$ that 
lies in the boundary ray of $S_n$ that converges to $b_-$,
and let $q_n$ be a critical point of $u_n$ that lies in
the boundary ray of $S_n$ that converges to $b_+$.
If the sequences $p_n$ and $q_n$ both converge to the vertex of $S$,
then $c_1=0$. If $\beta < \pi$, then we also have $c_0=0$.
\end{lem}

\begin{proof}
Let $\beta_n$ be the angle of the sector $S_n$. 
By performing rigid motions if necessary, we may assume without
loss of generality that the vertex $S$ and each of $S_n$ is $0$
and that the boundary rays of $S_n$ are $\theta= 0, \beta_n$. 
Thus, there exist sequences $r_n$ and $s_n$
so that $p_n = r_n$ and $q_n = s_n e^{i\beta_n}$.

From (\ref{eqn:Bessel-expansion-II}) we find that 
\begin{eqnarray*}
\partial_r u_n(r)
&=& 
c_0(n) \cdot 2r \cdot g_0'(r^2) ~
+~
c_1(n) \cdot \nu_n \cdot r^{\nu_n-1} \cdot g_{\nu_n}(r^2)~
+~
O \left(r^{\nu_n+1}  + r^{2 \nu_n-1} \right)
\\
\partial_r u_n \left(s \cdot e^{i \beta_n} \right) 
&=&  
c_0(n) \cdot 2s \cdot g_0'(s^2)~
-~
c_1(n) \cdot \nu_n \cdot s^{\nu_n-1} \cdot g_{\nu_n}(s^2)~
+~
O \left(s^{\nu_n+1}  + s^{2 \nu_n-1}  \right).
\end{eqnarray*}
Since $p_n = r_n$ and $q_n= s_n e^{i\beta_n}$ are critical points, the radial 
derivative of $u_n$ vanishes at these points, and hence 
\begin{eqnarray} 
0
&=& 
c_0(n) \cdot 2r_n \cdot g_0'(r_n^2) ~
+~
c_1(n) \cdot \nu_n \cdot r_n^{\nu_n-1} \cdot g_{\nu_n}(r_n^2)~
+~
O \left(r_n^{\nu_n+1}  + r_n^{2 \nu_n-1}  \right)
\label{eqn:distinct-sides-r}
\\
0
&=&  
c_0(n) \cdot 2s_n \cdot g_0'(s_n^2)~
-~
c_1(n) \cdot \nu_n \cdot s_n^{\nu_n-1} \cdot g_{\nu_n}(s_n^2)~
+~
O\left(s_n^{\nu_n+1}  + s^{2 \nu_n-1} \right).  
\label{eqn:distinct-sides-s}
\end{eqnarray}
Let $a_{\nu}(r)= 2 g_0'(r^2)/g_{\nu}(r^2)$.
Because, the functions $g_0'$ and $g_{\nu}$ are continuous and 
positive near zero, so is $a_{\nu}$. From 
(\ref{eqn:distinct-sides-r}) and (\ref{eqn:distinct-sides-s}) we find that
\begin{equation} \label{eqn:c_0-eqn}
c_0(n) 
\cdot 
\left( a_{\nu_n}(r_n) \cdot r_n^{2-\nu_n}~
+~
a_{\nu_n}(s_n) \cdot s_n^{2 - \nu_n} \right)~ 
=~ 
O \left(  r_n^2 + s_n^2~ +~ r_n^{\nu_n} + s_n^{\nu_n} \right),
\end{equation}
and 
\begin{equation} \label{eqn:c_1-eqn}
c_1(n) 
\cdot 
\left( \frac{r_n^{\nu_n-2}}{ a_{\nu_n}(r_n)}~
+~
\frac{s_n^{\nu_n-2}}{a_{\nu_n}(s_n)} \right)~ 
=~ 
O \left(  r_n^{\nu_n}+ s_n^{\nu_n}~ +~  r_n^{2\nu_n-2}+ s_n^{2\nu_n-2} \right).
\end{equation}
It follows from  (\ref{eqn:c_1-eqn}) that 
$c_1(n) = O(r_n^2 + s_n^2 + r_n^{\nu_n} + s_n^{\nu_n})$.
Since $\nu_n$ tends to $\nu>0$, we have $c_1= \lim_{n \to \infty} c_1(n) =0$.

It follows from  (\ref{eqn:c_0-eqn}) that 
$c_0(n)= O(r_n^{\nu_n} + s_n^{\nu_n} + r_n^{2\nu_n-2} + s_n^{2\nu_n-2})$.
If $\beta< \pi$, then there exists $\epsilon>0$ so that 
$\nu_n > 1+ \epsilon$ and for sufficiently large $n$.
Thus, for $n$ sufficiently large, we have 
$c_0(n)= O(r_n^{1+ \epsilon} + s_n^{1+ \epsilon}+ 
r_n^{2\epsilon}+s_n^{2\epsilon})$. 
Therefore, $c_0= \lim_{n \to \infty} c_0(n) =0$.
\end{proof}

\begin{lem}[Compare Lemma 9.3 \cite{J-M}] 
\label{lem:convergence-to-vertex-coefficient-to-zero}
Let $p_n$ be a critical point of $u_n$
and suppose that $p_n$ converges to the vertex of $S$. 
If $\beta< \pi/2$, then $c_0=0$.  
If $\beta > \pi/2$,  then $c_1=0$. 
\end{lem}

\begin{proof}
By performing rigid motions if necessary, we may assume without
loss of generality the boundary rays of $S_n$ are $\theta= 0$
and $\theta=\beta_n$. 
By Lemma \ref{lem:interior-crit-pt-convergence} passing
to a subsequence, and applying a reflection across $\theta=\beta_n/2$ 
if necessary, we may assume, without loss
of generality, that $p_n=r_n$ lies in the positive real axis.
As in the proof of Lemma \ref{lem:simult-conv-crit-pts} 
we have 
\begin{equation} \label{eqn:r-derivative}
0~
=~
c_0(n) \cdot 2r_n \cdot g_0'(r_n^2) ~
+~
c_1(n) \cdot \nu_n \cdot r_n^{\nu_n-1} \cdot g_{\nu_n}(r_n^2)~
+~
O \left(r^{\nu_n+1}  + r^{2 \nu_n-1} \right).
\end{equation}
If $\beta< \pi/2$, then there exists $\epsilon>0$ so 
that $\nu_n > 2 + \epsilon$ for $n$ sufficiently large. 
Hence, since $g_0'(0)\neq0$,
it follows from (\ref{eqn:r-derivative}) that 
$c_0(n) = O(r_n^{\epsilon})$. It follows that $c_0=0$.

From (\ref{eqn:r-derivative}), we have 
$
c_1(n)~
=~
O(r_n^{2-\nu_n})~
+~
O \left(r_n^2 + r_n^{\nu_n} \right).
$
If $\beta > \pi/2$, then  there exists $\epsilon>0$ so 
that $\epsilon< \nu_n < 2 - \epsilon$ for $n$ sufficiently large. 
Hence, since $g_{\nu}(0)\neq 0$,
it follows from (\ref{eqn:r-derivative}) that 
$c_1(n) = O(r_n^{\epsilon})$. Thus $c_1= \lim c_1(n)=0$.
\end{proof}

\begin{lem} \label{lem:two-crit-to-acute}
Suppose that $\beta \neq \pi/2$ and $\beta <\pi$.
Suppose that for each $n$ the sector $S_n$ is bounded by the rays 
$\theta= 0$ and $\theta =\beta_n$, and there exist $0< r_n \leq s_n$ 
such that $\partial_r u(r_n)=0$ and $\partial_r^2 u(s_n)=0$.  
If $s_n$ converges to zero as $n$ tends to infinity,
then $c_0=0=c_1$.
\end{lem}

\begin{proof}

Because $0<\beta< \pi$ and $\beta_n \to \beta$, there exists 
$\delta > 0$ such that
$\pi \cdot \delta <\beta_n < \pi \cdot (1+ \delta)^{-1}$ 
and hence $\delta^{-1} > \nu_n > \nu_n -1 > \delta$. 
From (\ref{eqn:Bessel-expansion-II}) we have 
\begin{eqnarray}
\partial_r u_n(r)
&=& 
c_0(n) \cdot 2r \cdot g_0'(r^2) ~
+~
c_1(n) \cdot \nu_n \cdot r^{\nu_n-1} \cdot g_{\nu_n}(r^2)~
+~
O \left(r^{\nu_n+1}  + r^{2 \nu_n-1} \right)
\label{eqn:r-derivativeII}
\\
\partial_r^2 u_n \left(s\right) 
&=&  
c_0(n) \cdot \left( 2 \cdot g_0'(s^2) + 4 s^2 \cdot g_0''(s^2) \right)~
+~
c_1(n) \nu_n  (\nu_n-1) s^{\nu_n-1}  g_{\nu_n}(s^2)~
+~
O \left(s^{\nu_n}  + s^{2 \nu_n-2}  \right). \nonumber
\end{eqnarray}
Let 
\[ 
a_{\nu}(r)~
=~
\frac{2 g_0'(r^2)}{ \nu \cdot g_{\nu}(r^2)}
\]
and 
\[ 
b_{\nu}(s)~ 
=~
\frac{2 g_0'(s^2) + 4 s^2 \cdot g_0''(s^2)}{ \nu \cdot g_{\nu}(s^2)}.
\]
Because $g_{\nu}$ and its derivatives are positive and continuous
for $r$ near zero, the functions $a_{\nu}$ and $b_{\nu}$ are also
positive and continuous for small $r$. Note that 
$a_{\nu}(0)/b_{\nu}(0) =1$.

Since $\partial_r u(r_n)=0$ and $\partial_r^2 u(s_n)=0$ we find 
from (\ref{eqn:r-derivativeII}) that
\begin{eqnarray}
0
&=&
c_0(n) \cdot a_{\nu_n}(r_n) \cdot r_n^{2-\nu_n}~
+~
c_1(n)~
+~
O \left(r_n^{2}  + r_n^{\nu_n} \right)
\label{eqn:a-b}
\\
0 
&=&  
c_0(n) \cdot \frac{b_{\nu_n}(s_n)}{ \nu_n-1} \cdot s_n^{2 -\nu_n}~
+~
c_1(n)~ 
+~
O \left(s_n^{2}  + s_n^{\nu_n}  \right). \nonumber
\end{eqnarray}
By subtracting we have 
\begin{equation} 
\label{eqn:c_0-bound}  
c_0(n)
\cdot 
\left(
~a_{\nu_n}(r^{\nu_n}) \cdot r^{2 -\nu_n}~
-~
\frac{b_{\nu_n}(s_n)}{ \nu_n-1} \cdot s_n^{2 -\nu_n}
\right)~
=~
O\left(
r_n^{2}  + r_n^{\nu_n}
+
s_n^{2}  + s_n^{\nu_n} 
\right)
\end{equation}

Suppose $\beta > \pi/2$, then $\nu<2$ and so since $\nu_n \to \nu$
there exists $\epsilon>0$ so that for sufficiently large $n$
\begin{equation}
\label{eqn:b-over-a}
\frac{1}{\nu_n-1} \cdot \frac{b_{\nu_n}(s_n)}{a_{\nu_n}(r_n)}~ 
\geq~ 
1+ \epsilon.
\end{equation}
Since $r_n \leq s_n$, we have $r_n^{2-\nu_n} \leq s_n^{2-\nu_n}$.
Therefore, from (\ref{eqn:c_0-bound}) we find that 
\[  
c_0(n) \cdot (-\epsilon) \cdot a(r_n) \cdot s^{2-\nu_n}~
=~
O\left(
r_n^{2}  + r_n^{\nu_n}
+
s_n^{2}  + s_n^{\nu_n} 
\right)
\]
Thus, since $r_n \leq s_n$ we find that
$c_0(n) = O(r_n^{\nu_n} + s_n^{\nu_n}+ r_n^{2\nu_n-2} + s_n^{2\nu_n-2})$,
and hence 
\begin{equation}
c_0(n)~
=~
O \left(r_n^{1+\delta} + s_n^{1+\delta}+ r_n^{2\delta} + s_n^{2 \delta}
\right).
\end{equation}
Therefore, $c_0=\lim c_0(n)=0$.

Suppose $\beta < \pi/2$. Then since $\nu_n \to \nu >2$, there exists
$\epsilon>0$ so that for sufficiently large $n$
\begin{equation}
\label{eqn:a-over-b}
(\nu -1) \cdot \frac{ a_{\nu_n}(r_n)}{b_{\nu_n}(s_n)}~ 
\geq~
1+\epsilon.
\end{equation}
Since $r_n^{2 -\nu_n} \geq s^{2-\nu_n}$, from (\ref{eqn:c_0-bound})
one deduces that $c_0=0$ in this case by arguing in a similar
manner.

To show that $c_1=0$, we argue similarly. From (\ref{eqn:a-b})
we find that 
\begin{eqnarray*}
0
&=&
c_0(n)~ 
+~
c_1(n) \cdot \frac{r_n^{\nu_n-2}}{a_{\nu_n}(r_n)}~
+~
O \left(r_n^{2}  + r_n^{\nu_n} \right)
\\
0 
&=&  
c_0(n)~  
+~
c_1(n) 
\cdot  
\frac{ (\nu_n-1) \cdot s_n^{\nu_n-2}}{b_{\nu_n}(s_n)}~
+~
O \left(s_n^{2}  + s_n^{\nu_n}  \right). \nonumber
\end{eqnarray*}
and hence by subtracting
\[ 
c_1(n)
\cdot
\left(
\frac{r_n^{\nu_n-2}}{a_{\nu_n}(r_n)}~
-~
\frac{ (\nu_n-1) \cdot s_n^{\nu_n-2}}{b_{\nu_n}(s_n)}
\right)
=~
O \left(r_n^{2}  + r_n^{\nu_n} + s_n^{2}  + s_n^{\nu_n}  \right).
\]
Now argue as was done to show that $c_0=0$.
In particular, in the case  $\beta < \pi/2$
use (\ref{eqn:a-over-b}), and in the case $\beta> \pi/2$
use (\ref{eqn:b-over-a}).

\end{proof}

\begin{coro}  \label{coro:two-crit-to-acute}

Suppose $\beta < \pi$ and $\beta \neq \pi/2$. 
Suppose that for each $n$, the points $p_n$ and $q_n$ are
distinct critical points.  If $p_n$ and $q_n$ both converge to 
the vertex of $S$, then $c_0=0=c_1$.
\end{coro}

\begin{proof}
By applying rigid motions we may assume that $S_n$
is bounded by the rays $\theta= 0$ and $\theta= \beta_n$. 
By Lemma \ref{lem:interior-crit-pt-convergence}
and Lemma \ref{lem:simult-conv-crit-pts}, it suffices to assume
that $p_n$ and $b_n$ lie in the same boundary ray, and by 
reflecting if necessary about $\theta=\beta_n/2$, we may 
assume that both $p_n$ and $q_n$ are real.
By relabeling we may assume that $p_n < q_n$.
By assumption $\partial_r (p_n)= 0 = \partial_r (q_n)$,
and so Rolle's theorem implies that there exist $s_n$
such that $p_n \leq s_n \leq q_n$ and $\partial_r^2(s_n)=0$.
The claim now follows from Lemma \ref{lem:two-crit-to-acute}.
\end{proof}

\begin{coro} \label{coro:accum-at-vertex-Bessel-vanish}
Let $S$ be a sector with angle $\beta < \pi$ and not equal to $\pi/2$,
and let $u: S \to \Rbb$ be a Neumann eigenfunction.
If the vertex $v$ is an accumulation point of the critical points of $u$,
then $c_0=0=c_1$.
\end{coro}

\begin{proof}
Apply Corollary \ref{coro:two-crit-to-acute} with $S_n=S$ and $u_n=u$.
\end{proof}

\begin{lem} \label{lem:degenerate-converge-vertex}
Suppose $\beta < \pi$ and $\beta \neq \pi/2$.
If $p_n$ is a degenerate critical point of $u_n$ that converges
to the vertex of $S$, then $c_0=0=c_1$. 
\end{lem}

\begin{proof}
By applying rigid motions we may assume that $S_n$ 
is bounded by $\theta=0$ and $\theta=\beta_n$.
By Lemma \ref{lem:interior-crit-pt-convergence},
by passing to a subsequence, and by
applying a reflection across $\theta=\beta_n/2$ if necessary,
we may assume that $p_n$ lies in the boundary ray $\theta=0$.
That is, $p_n=r_n>0$ and $\partial_r u_n(r_n)=0$.

Since $u_n$ satisfies Neumann conditions along the real axis,
and $p_n$ is a degenerate critical point we have either 
$\partial_x^2 u_n(r_n)=0$ or $\partial_y^2 u_n(p_n)=0$.
If $\partial_x^2 u_n(r_n)=0$, then Lemma \ref{lem:two-crit-to-acute}
with $s_n=r_n$ implies the claim.

Suppose then that $\partial_y^2 u_n(p_n)=0$.
Along the ray $\theta=0$ we have 
$\partial_y^2 = r^{-1} \cdot \partial_r + r^{-2} \cdot \partial_{\theta}^2$.
Since $\partial_r u_n(r_n)=0$, we have 
$\partial_y^2 u_n(r_n) = \partial_{\theta}^2 u_n(r_n) $,
and so  
\[ 
0~=~
\left(\partial_{\theta}^2 u_n \right)(r_n)~ 
=~
- c_1(n) \cdot \nu_n^2\cdot r^{\nu_n} \cdot g_0(r_n^2)~
+~
O(r_n^{2 \nu_n}).
\] 
Since $u_n$ satisfies the first equation in (\ref{eqn:r-derivative})
we find that
\[ 
0~=~
\left(\partial_y^2 u_n\right)(r_n)~
=~
2 c_0(n) \cdot g_0'(r_n^2)~
+~
O\left(r_n^{\nu_n} + r_n^{2 \nu_n -2} \right).
\]
Since $\beta < \pi$, there exists $\epsilon>0$
so that $\nu_n > 1+ \epsilon$ for suficiently large $n$. 
Since $g_0$ and its derivative do not vanish at zero,
it follows that 
$c_0 = \lim_{n \to \infty} c_0(n) = 0$
and $c_1 = \lim_{n \to \infty} c_1(n) = 0$.

\end{proof}

\begin{remk}
Note that in the proof of Lemma \ref{lem:degenerate-converge-vertex}
we used the condition $\beta \neq \pi/2$ only in the case 
that $\partial_r^2 u_n(p_n)=0$. Indeed, the proof shows
that if $\partial_{\theta}^2 u_n(p_n)=0$, then $c_0=0=c_1$
even if $\pi=\beta/2$.
\end{remk}

\section{A Poincar\'e-Hopf formula for critical points of eigenfunctions on a polygon}
\label{sec:topology}
In this section, we provide a variant of the 
the classical-Poincar\'e Hopf index theorem for the gradient 
of Laplace eigenfunctions on a planar polygonal domain $P$. 
The discussion will focus on eigenfunctions
satisfying Neumann boundary conditions, 
but the methods apply to give variants in the 
contexts of Dirichlet and mixed boundary conditions.

Let $u:P \to \Rbb$ be a Neumann eigenfunction and let $p \in P$. 
Suppose that there exists a deleted disc neighborhood $\dot{D}$ of $p$
that contains no zeros of $\nabla u$. Then the closure of each component of 
$\dot{D} \cap \{z :u(z)=u(p)\}$ is an arc.\footnote{If the 
closure of some component were a loop, then the the loop would bound a disk that 
contained a critical point.} If such an arc contains $p$,
then we will say that the arc {\em emanates} from $p$.
Let $n$ be the number of arcs in $\{z :u(z)=u(p)\}$
that emanate from $p$, and define 
\[ 
\ind(u,p)~ =~ 
\left\{
\begin{array}{cc}
1~ -~ \frac{1}{2} \cdot n  &  \mbox{ if } p \in P^{\circ} \\
1~ -~ n  &  \mbox{ if } p \in  \partial P.
\end{array}
\right.
\]
Note that if $\ind(u,p) \neq 0$ and $p$ is not a vertex of $P$, 
then $p$ is a critical point.\footnote{The converse is not true,
namely there may be critical points with index equal to zero. See \S 
\ref{sec:zero-index}.}
If $p$ is a vertex and $\ind(u,p) \neq 0$, then we will regard 
 $p$ as a critical point of  $u$.

Let $\chi(S)$ denote the Euler characteristic of a surface $S$.
Let $\crit(u)$ denote the set of critical points of $u$ including 
the vertices $v$ such that $\ind(u,v) \neq 0$.
The following is a variant of the classical Poincar\'e-Hopf 
formula  \cite{Taylor}.
\begin{prop}[Index formula]  
\label{prop:index-thm}
Let $u: P \to \Rbb$ be a Neumann eigenfunction such that the set 
$\crit(u)$ is finite. 
\[
2 \cdot \chi(P)~
=~ 
\sum_{p\, \in\, \crit(u) \cap P^{\circ}}
2 \cdot \index(u,p)~
+~
\sum_{p\, \in\, \crit(u) \cap \partial P}
\index(u,p).
\]
\end{prop}

\begin{proof}

Let $DP$ be the `double of $P$', the 
closed surface without boundary obtained by gluing two disjoint copies of $P$
along their respective boundaries. The surface $DP$  
has a natural real-analytic structure on the complement of the set $C$ of 
`cone points' corresponding to the vertices of $P$. 
Because $u$ is a Neumann eigefunction, $u$ extends 
to a real-analytic 
function $\tu: DP \setminus C \to \Rbb$ that is invariant 
under the isometric involution that exchanges the two copies of $P$. 
For each $p \in DP$, we define $\ind(\tu,p)= 1- \frac{n}{2}$. 
Because $u$ is a Neumann eigenfunction, we find that 
$\ind(\tu,p)= \ind(u,p)$ for $p \in P$ (including vertices).

Let $A$ be the union of the level sets of $\tu$ that contain critical 
points of $\tu$. The complement of $A$ consists of topological annuli, 
and hence, by the Euler-Poincar\'e formula, $\chi(DP)=\chi(A)$. 
On the other hand, the number of edges in $A$
equals $\frac{1}{2} \sum_p n_p$ where $n_p$ is the valence of 
the graph $A$ at $p$. It follows that 
$\chi(DP)= \sum_{\crit(\tu)} \ind(u,p)$ 
where $\crit(\tu)$ includes  $p \in C$
such that $\ind(\tu,p) \neq 0$.
We have $\chi(DP) = 2 \cdot \chi(P)$ and for every 
interior critical of $u$ we have two critical points 
of $\tu$ with the same index. The claimed formula follows. 
\end{proof}

\begin{remk}
There are also variants of Proposition \ref{prop:index-thm}
in the contexts of Dirichlet and mixed boundary conditions. 
For example, if $u$ satisfies Dirichlet conditions,
then formula (\ref{prop:index-thm}) holds true
if one redefines $\ind(u,p)= 2-k$ for each for $p \in \partial P$. 
\end{remk}

\begin{remk} 
\label{remk:PH-classical}
The classical Poincar\'e-Hopf theorem applies to a smooth 
vector field $X$ on an oriented closed surface $S$ that 
has finitely many critical points.
If $\gamma$ is a simple oriented loop that encloses 
at most one zero $p$ of $X$, then the restriction of $X/|X|$
to $\gamma$ defines a map from the unit circle to itself. 
The index of $X$ at $p$ is the degree of this self-map of the circle.
(See, for example, \cite{Taylor} \S 1.10.) 
If $X = \nabla f$, then this index equals $1- \frac{k}{2}$ 
where $k$ is the number of components of $f^{-1}(f(p)) \setminus \{p\}$.
In the context of a vector field $X$, the Poincar\'e-Hopf index formula 
gives that the sum of the indices of the zeros of  $X$ 
equals the Euler characteristic of $S$.
\end{remk}

\begin{ass}
In what follows we will assume that each 
critical point $p$ is isolated and so the 
index $\ind(u,p)$ is well-defined. 
\end{ass}

In \cite{J-M-arc}, we show that rectangles are the only 
simply-connected polygons whose second Neumann eigenfunctions 
have infinitely many critical points.
Hence the assumption reduces to the assumption that the polygon
is not a rectangle in the simply-connected case.

\begin{prop} \label{prop:local-ext-index-1}
The point $p \in P$ is a local extremum if and only if $\ind(u,p) =1$. 
\end{prop}

\begin{proof}

We have $\ind(u,p)=1$ if and only if there exists a punctured disc
neighborhood $\dot{D}$ of $p$ so that $u(z) \neq u(p)$ for 
each $z \in \dot{D}$. Since $u$ is continuous, we have either 
$u(z) > u(p)$ for all $z \in \dot{D}$
or $u(z) < u(p)$ for all $z \in \dot{D}$. This occurs if and only 
if $p$ is a local extremum of $u$.
\end{proof}

Suppose that $v$ is a vertex of $P$ that is not a limit point of 
the zeros of $\nabla u$.    
The index $\ind(u,v)$ is determined 
by the Bessel expansion (\ref{eqn:Bessel-expansion-II}) of $u$ near $v$. 

\begin{lem} 
\label{lem:vertex-index-1}
Let $P$ be a polygon, let $v$ be a vertex of $P$ with angle $\beta$, 
and let $u$ be a Neumann eigenfunction on $P$.
Let $k \geq 1$ be the smallest positive integer so that $c_k \neq 0$
and suppose that $v$ is a critical point of $u$. 
\begin{enumerate}[label=(\roman*)]

\item If $u(v)=0$ or $ \beta > k \cdot \frac{\pi}{2}$, then $\ind(u,v)= 1-k$.

\item If $u(v) \neq 0$ and $\beta < k \cdot \frac{\pi}{2}$, then $\ind(u,v)=1$.

\item If $u(v) \neq 0$ and $\beta = k \cdot \frac{\pi}{2}$, 
then $1-k \leq \ind(u,v) \leq 1$. 
\end{enumerate}
In particular, if $\beta \neq \pi/2$ or $3\pi/2$, then 
$\ind(u,v)$ equals either $1$ or $1-k$.
\end{lem}

A similar statement can be derived in the cases of Dirichlet or mixed boundary conditions. 

\begin{proof}
Without loss of generality, $v =0$ and the sides adjacent to $v$ 
bound the sector $0 < \theta < \beta$.

If $u(0)=0$ or $\beta > k \cdot \frac{\pi}{2}$, then from (\ref{eqn:Bessel-expansion-II}) there exist $b \neq0$ and $a$
so that
\[ 
u \left(r \cdot e^{i\theta} \right)~ 
=~  
a~ 
+~ 
b \cdot r^{k\nu} \cdot \cos(k \nu \theta) ~  +~ o\left(r^{k \nu} \right).
\]
Using, for example, the implicit function theorem, one finds that 
there exists a disk neighborhood $D$ of $0$ such that 
$D\cap u^{-1}(u(v)) \setminus \{v\}$
consists of $k$ arcs each with an endpoint at $v$.
It follows that $\ind(u,0) =1-k$. 

Suppose $u(v)\neq 0$ and $\beta < k \cdot \frac{\pi}{2}$.
Then $k \cdot \nu > 2$ and hence from (\ref{eqn:Bessel-expansion-II})
we find that
\[ 
u \left(z \right)~ 
=~  
a~ +~  b \cdot r^{2}~  +~ o(r^2)
\]
where $a \neq 0 \neq b$. Hence $v$ is a local extremum of $u$,
and so by Proposition \ref{prop:local-ext-index-1}, $\ind(u,v)=1$.

If $u(v)\neq 0$ and $\beta = k \pi/2$, then from (\ref{eqn:Bessel-expansion-II})
we have
\[ 
u \left(r \cdot e^{i\theta} \right)~ 
=~  
a~ 
+~
r^{2} \left( b~ + c \cdot \cos(k \nu \theta)\right)~  
+~ o\left(r^{2} \right)
\]
where $a$, $b$ and $c$ are nonzero constants.
If $b = -c$, then $\ind(u,v)$ will depend on 
the $o(r^2)$ error term. In this case $1-k \leq \ind(u,v) \leq 1$. 
\end{proof}

\begin{coro} \label{coro:c_1-zero-nonzero-index}
Suppose $\beta$ is not a multiple of $\pi/2$.

\begin{enumerate}
\item If $c_1 = 0$, then $\ind(u,v) \neq 0$.
\item If $\beta > \pi/2$, then $c_1=0$ if and only if $\ind(u,v) \neq 0 $.
\end{enumerate}

\end{coro}

\begin{proof}
Let $k$ be as in the statement of Lemma \ref{lem:vertex-index-1}.
 If $c_1=0$, then $k>1$, and hence Lemma \ref{lem:vertex-index-1}
implies that $\ind(u,v) \neq 0$.
If $\beta> \pi/2$ and $c_1 \neq 0$, then $k=1$ and 
we are in case (i) of Lemma \ref{lem:vertex-index-1}.
Thus, $\ind(u,v)=0$. 
\end{proof}

\begin{remk}
\label{remark-k-nu-2}
In the case that $\beta$ is a multiple of $\pi/2$ and $u(v) \neq 0$,  
part (iii) of Lemma \ref{lem:vertex-index-1} provides only an inequality for $\ind(u,v)$. Yet,

one can determine the index in finitely many steps.
In particular if $k \cdot \nu =2$, then 
\[
u(z)-u(v)~
=~
r^2 
\cdot 
\left( 
a~
+~
\cos( 2 \theta) 
\right)~
+~
o(r^2)
\]
where $a= (c_0 \cdot g_0'(0))/(c_k \cdot g_2(0))$.
If $|a|>1$, then $\ind(u,v)=1$  and if $|a|<1$, then $\ind(u,v)=1-k$.
If $|a|=1$, then 
by considering more terms of the Bessel expansion,
one can identify $\ind(u,v)$. 
\end{remk}

If $p$ is an isolated critical point of an eignfunction $u$
that lies in the interior of a polygon $P$, 
then $\ind(u,p)$ equals the degree of the mapping 
$\nabla u /|\nabla u| \circ \gamma$ as described in 
Remark \ref{remk:PH-classical}. If $p$ lies in the interior
of a side of $P$, then one may reflect a Neumann eigenfunction 
across the side to $\tu$, and then find that $\ind(u,p)$ 
equals degree of the map $\nabla \tu /|\nabla \tu| \circ \gamma$.

If $p$ is a vertex, we may also interpret $\ind(u,v)$ in terms of
the degree of the self-map of the circle induced by a vector field.
Indeed, let $D$ be a disc centered  at $p$ that intersects no sides of $P$ 
other than the side(s) adjacent to $p$ and so that $\oD \setminus \{p\}$ 
contains no critical points of $u$ other than possibly $p$.
By applying a rigid motion we way assume that $p=0$
and that $D \cap P$ lies in the sector $S$ bounded
by the rays $\theta=0$ and $\theta= \beta$. 
Moreover, by rescaling if necessary,
we may assume that $D$ is the unit disk.
The map $z \mapsto z^{\frac{\beta}{\pi}}$ maps 
$H= \{ z \in \Cbb: |z|<1 \mbox{ and } y>0\}$ to 
the sector $D \cap P$. In particular, the function 
$w(z) = u\left(z^{\frac{1}{\nu}}\right)$ is defined on $H$.
If $u$ is given by (\ref{eqn:Bessel-expansion-II}), then 
\begin{equation}
\label{eqn:w}
w\left(r \cdot e^{i \theta} \right)~
=~
\sum_{j=0} c_j \cdot r^{j} 
\cdot g_{j \cdot \nu}\left(r^{\frac{2}{\nu}} \right) 
\cdot
\cos(j \cdot \theta).
\end{equation}
We may extend $w$ smoothly to $D \setminus \{0\}$
by setting $w(\oz)= w(z)$.  

\begin{lem}
\label{lem:degree-index-vertex}
The degree of the restriction
of $ \frac{\nabla w}{|\nabla w|}$ to the unit circle
equals $2 \cdot \sum \ind(u,q)$ where the sum is over 
critical points $q$ of $u$ that lie $D$.
\end{lem}

\begin{proof}
Suppose $q \neq 0$ is a critical point of $u$.
If $q$ lies in the interior of $P$, then $q$ corresponds 
to two critical points $q_+$ and $q_-$ of $w$ which have the 
same indices as $q$. By Remark \ref{remk:PH-classical},
since $w$ is smooth at $q_{\pm}$,
the index $\ind(w,q_{\pm})$ equals 
the degree of the restriction of $\nabla w/|\nabla w|$
to a small circle centered at $q_{\pm}$. 
If $q \neq 0$ lies on the boundary of $P$, then $q$ 
corresponds to a single critical point $q'$ of $w$,
and $\ind(w,q')$ equals the degree of $\nabla w/|\nabla w|$
on a small circle centered at $q'$. By choosing disjoint
circles, and applying a standard argument\footnote{
See for example, the proof of Proposition 20.2 in \cite{Taylor}},
we find that it suffices to assume that $u$ 
has no critical points in $\oD \setminus \{0\}$.

Since $u$ has no critical points in $\oD \setminus \{0\}$, 
the function $w$ has no critical points in $\oD \setminus \{0\}$ 
In particular, since $\partial_{\theta}w$ vanishes on the 
real line, it follows that $w(z) \neq w(0)$ for each $z \neq 0$ 
on the real axis. Hence the 
number of arcs in $\{z: w(z)=w(0)\}$ that emanate from $0$
equals twice the number of arcs in $\{z: u(z)=u(0)\}$ that
emanate from $0$. Thus, to complete the proof, it suffices to show that 
the degree of $ \frac{\nabla w}{|\nabla w|} \circ \gamma$ 
where $\gamma$ is the unit circle equals $1-n/2$ where $n$ 
is the number of arcs of 
$\{z: w(z)=w(0)\}$ that emanate from $0$.

Let $h(z):= w(z)-w(0)$ and let $k$ be the smallest positive interger 
such that $c_k \neq 0$. Then 
\[ 
\partial_{\theta} 
h\left(r \cdot e^{i \theta} \right)~
=~
-
c_k 
\cdot 
k 
\cdot 
r^{k-1} 
\cdot 
g_{k \cdot \nu}\left(r^{\frac{2}{\nu}} \right) 
\cdot 
\sin(k \cdot \theta)~
+~ 
O\left(r^{k}\right), 
\]
and so there exists $r_0>0$ so that if $0<r  \leq r_0$, then the set 
$\{\theta\, :\, \partial_\theta h \left( r \cdot e^{i \theta} \right) = 0\}$ 
consists of exactly $2k$ elements, $\theta_0(r), \ldots, \theta_{2k-1}(r)$.
Using the implicit function theorem, we find that, for each $j$, the map
$r \mapsto \theta_j(r)$ is smooth.  By relabeling
if necessary, we may assume that $\lim_{r \to 0} \theta_j(r)= j \cdot \pi/k$.
The function $h$ has no critical points 
in $\oD \setminus \{0\}$, and so the degree of 
$\nabla h/|\nabla h| \circ \gamma$ equals the 
degree of the map $\nabla h/|\nabla h| \circ \gamma_0$
where $\gamma_0$ is the standard counterclockwise parameterization of $r=r_0$.

Choose a homeomorphism $\psi: \oD \to \oD$ 
that is isotopic to the identity map,
that is smooth away from $0$, and that maps each ray $\theta= j \cdot \pi/k$
to the arc $\theta_j$. Then if we define $\tilh(z)=h \circ \psi$, then 
the degree of $\nabla \tilh/|\nabla \tilh| \circ \gamma_0$ equals
the degree of $\nabla h/|\nabla h| \circ \gamma_0$ and 
$\ind(\tilh,0)= \ind(h,0)$.

Let $j \in \{1, \ldots, 2k\}$ and let $\theta_j:=j \pi/k$.
Since $\tilh$ has no critical points in $\oD \setminus \{0\}$,
the mean value theorem implies that 
$r \mapsto |\tilh\left(r e^{i\cdot \theta_j} \right)|$ is
strictly increasing and thus 
$\tilh\left(r e^{i\cdot \theta_j}  \right) \neq 0$
for each $r \in (0, r_0]$. Let $\epsilon_{j} \in \{+1,-1\}$ denote
the sign of the function 
$r \mapsto \tilh\left(r e^{i\cdot \theta_j} \right)$.
Note that $\epsilon_j$ is also the sign of 
$\partial_r\tilh\left(r e^{i\cdot \theta_j} \right)$.

The number arcs in $\{z\, :\, \tilh(z)=0\}$ emanating 
from $0$ equals the number of $j \in \{1, \ldots, 2k\}$
such that $\epsilon_{j} \neq \epsilon_{j+1}$.  
Indeed, for each fixed $r$, the restriction of 
$\theta \mapsto \tilh\left(r e^{i\theta} \right)$ to the interval 
$I_j := [\theta_j, \theta_{j+1}]$ 
is monotone, and hence 
$\theta \mapsto \tilh\left(r e^{i\theta} \right)$ 
assumes the value $0$ at most once,
and it assumes the value $0$ if and only if 
$\epsilon_j \neq \epsilon_{j+1}$.
In other words, $\ind(\tilh, 0)$ equals the number
of $j$ such that $\epsilon_j \neq \epsilon_{j+1}$.

To compute the degree of $\nabla \tilh/|\nabla \tilh|\circ \gamma_0$, we first 
regard this map as a map $X: \Rbb/2\pi \Zbb \to \Rbb/2\pi \Zbb$.
In particular, 
for each $\theta$ there exists a unique $X(\theta)$ so that
$\nabla \tilh/|\nabla \tilh|(r \cdot e^{i \theta})$
corresponds to the point $e^{i X(\theta)}$ in the unit circle.
In other words, 
$X(\theta)$ is the angle  between
the vector $\partial_x$ and $\nabla \tilh$
measured counterclockwise.

We have 
$\nabla \tilh= \partial_r \tilh \cdot \partial_r 
      + r^{-2} \cdot \partial_{\theta} \tilh \cdot \partial_{\theta}$.
Since $\partial_{\theta} \tilh(re^{\theta_j})=0$ 
we have 
$\nabla \tilh= \partial_r \tilh \cdot \partial_r$, and so
\[
X(\theta_j)~
=~
\left\{
\begin{array}{cc}
\theta_j \mod 2 \pi & \mbox{ if } \epsilon_j=+1, \\
\theta_j + \pi \mod 2 \pi & \mbox{ if } \epsilon_j =-1.
\end{array}
\right.
\]
We also have
$\partial_{\theta}\tilh(r_0 e^{i \theta}) > 0$ 
if and only if  $X(\theta) \in (\theta, \theta + \pi) \mod 2 \pi$,
and $\partial_{\theta}\tilh(r_0 e^{i \theta}) < 0$ if and only if
$X(\theta) \in (\theta-\pi, \theta)$ 
In particular,
we have either $X(\theta) \in [\theta, \theta+\pi]$ for each 
$\theta \in I_j$ or 
$X(\theta) \in [\theta-\pi, \theta]$
for each 
$\theta \in I_j$.

If $\epsilon_j = +1 = \epsilon_{j+1}$, then
$X(\theta_j) =\theta_j$ and $X(\theta_{j+1}) =\theta_{j+1}$
and  either $\theta \leq X(\theta) \leq \theta+\pi$
for each $\theta \in I_j$
or $\theta -\pi \leq X(\theta) \leq \theta$
for each $\theta \in I_j$.
It follows that the restriction of
$X$ to $I_j$ is homotopic to the identity map rel endpoints.
Similarly, if $\epsilon_j = -1 = \epsilon_{j+1}$,
then the restriction of $X$ to $I_j$ is
homotopic to the identity map rel endpoints.

If $\epsilon_j = -1$ and $\epsilon_{j+1}=+1$,
then $\partial_{\theta} \tilh(r_0e^{i \theta}) \geq 0$ 
for each $\theta \in I_j$, and so 
$X(\theta) \in [\theta, \theta+ \pi] \mod 2\pi$.
We also have $X(\theta_j)= \theta_j+ \pi \mod 2 \pi$ and 
$X(\theta_{j+1})= \theta_{j+1} \mod 2 \pi$.
It follows that $X$ is homotopic rel endpoints to the
linear map $Y^+_j:I_j \to \Rbb/2 \pi \Zbb$ defined by 
\[
Y_j^+(\theta)~
=~
(1-k ) \cdot \left( \theta - \theta_j \right)
+~
\theta_j~
+~
\pi
\mod 2 \pi.
\]
Similarly, if $\epsilon_j = +1$ and $\epsilon_{j+1}=-1$,
then one finds that $X|_{I_j}$ is homotopic rel endpoints to 
the map $Y^-_j:I_j \to \Rbb/2 \pi \Zbb$ defined by
\[
Y^-_j(\theta)~
=~
(1-k ) \cdot \left( \theta - \theta_j \right)
+~
\theta_j  \mod 2 \pi.
\]
Using the identity map on $I_j$ when $\epsilon_j = \epsilon_{j+1}$
and the maps $Y_j^+$ and $Y_j^-$  when $\epsilon_j \neq \epsilon_{j+1}$,
one constructs a piecewise linear map $Y: \Rbb/2\pi \Zbb \to \Rbb/2\pi \Zbb$ 
that is homotopic to $X$. An elementary argument shows that $Y$ is in turn
homotopic to the map $Z$ defined by $Z(\theta) =(1-\frac{n}{2}) \cdot \theta 
\mod 2 \pi$ where $n$ is the number of $j$ such that 
$\epsilon_j \neq \epsilon_{j+1}$.  The claim follows.

\end{proof}

\begin{remk}
If the angle $\beta$ at $v$ is not a multiple of 
$\pi/2$, then $k \cdot \nu \neq 2$, and the proof  
of Lemma \ref{lem:degree-index-vertex} can be significantly shortened.
Indeed, one can use expansion (\ref{eqn:w}) as in the proof of 
Lemma \ref{lem:vertex-index-1}.
However, if $\beta=\pi/2$ or $3 \pi/2$, then using  
expansion (\ref{eqn:w}) is more cumbersome. See Remark 
\ref{remark-k-nu-2}.
\end{remk}

Let $P_k$ be a sequence of $n$-gons and let $P$ be an $n$-gon. 
We will say that $P_k$ converges to $P$ if and only if there exists a
sequence of homeomorphisms $\phi_k: \oP \to \oP_k$ 
that are
$C^2$ diffeomorphisms on the complement of the vertices 
such that $\phi_k$ converges uniformly 
in $C^2$ to the identity map on each compact subset of $\oP$ 
that does not include the vertices.
Given continuous functions $u_k: P_k \to \Cbb$ 
and $u: P \to \Cbb$, we will say that $u_k$ converges to $U$
if and only if $u_k \circ \phi_k$ converges to $u$.

\begin{prop}[Stability of the total index]
\label{prop:local-number-of-critical-points}
Suppose that $P_n$ is a sequence of polygons that 
converges to $P$, and suppose that $u_n:P_n \to \Rbb$ 
is a sequence of Neumann eigenfunctions that converge 
to a Neumann eigenfunction $u: P \to \Rbb$.
Let $p \in P$ and suppose that $D \subset \Cbb$ is an open
disk neighborhood of $p$ such that $\partial D$  contains no zeros of $\nabla u$.
Let $A$ (resp. $A_n$) denote the set of critical points of $u$ (resp. $u_n$)
that lie in $D$. If $A$ and $A_n$ are finite, then
for each $n$ sufficiently large
\[  
\sum_{q \in A}\,
\ind(u,q)~ 
=~
\sum_{q \in A_n}\,
\ind(u_n, q).
\]
\end{prop}

\begin{proof}
The gradient $\nabla u_n$ converges to $\nabla u$, 
and so the sets $A_n$ converges to $A$.

First we suppose that $p$ lies in the interior of $P$.
Let $\gamma$ be a counterclockwise parameterization of $\partial D$. 
By Proposition 20.2 in \cite{Taylor}, we have that 
$\sum_{q \in A} \ind(u,q) = \deg( \nabla u/|\nabla u| \circ \gamma )$ and 
$\sum_{q \in A_n} \ind(u_n,q) = \deg( \nabla u_n/|\nabla u_n| \circ \gamma )$.
But the vector field $\nabla u_n/|\nabla u_n| \circ \gamma$ converges
to $ \deg( \nabla u/|\nabla u| \circ \gamma )$, and hence the degrees 
converge. Since the degree is an integer invariant, the degrees
coincide for all sufficiently large $n$.

If $p$ lies on the boundary of $P$, the we apply Lemma \ref{lem:degree-index-vertex}.
We have $\sum_{q \in A} \ind(u,q) = \deg( \nabla w/|\nabla w| \circ \gamma )$ and 
$\sum_{q \in A_n} \ind(u_n,q) = \deg( \nabla w_n/|\nabla w_n| \circ \gamma )$
where $w$ and $w_n$ are constructed as in (\ref{eqn:w}).  
Since $w_n$ converges to $w$, the degree of $\nabla w_n/|\nabla w_n| \circ \gamma$
converges to the degree of $\nabla w/|\nabla w| \circ \gamma$.
The claim the follows from Lemma \ref{lem:degree-index-vertex}.
\end{proof}

\begin{lem} 
\label{lem:top-stable-implies-stable}
Let $P_n$ be a sequence of polygons that converges to a polygon $P$
and let $u_n: P_n \to \Rbb$ be a sequence of Neumann eigenfunctions
that converge to a Neumann eigenfunction $u: P \to \Rbb$. 
If $u$ has finitely many nonzero index critical points, 
then there exists $N$ such that if $n>N$, then the number
of nonzero index critical points of $u_n$ is greater than or 
equal to the number of critical points of $u$.
\end{lem}

\begin{proof}
Let $p$ be a nonzero index critical point of $P$. In particular, 
$p$ is isolated, and so there exists a disk neighborhood
$D_p$ of $p$ that contains no critical points of $u$ other than $p$.  
Since $\ind(u,p) \neq 0$, Proposition \ref{prop:local-number-of-critical-points}
implies that, for sufficiently large $n$,
at least one nonzero index critical points of $u_n$ lies 
in $D_p$.  Since the various disks $D_p$ are disjoint, 
the claim follows.
\end{proof}

\begin{lem}  \label{lem:two-stable-on-boundary}
Let $u$ be a nonconstant Neumann eigenfunction on a polygon $P$.
If the set of critical points of $u$ is discrete, then 
each local extremum $p$ of the restriction 
$u|_{\partial P}$ is a critical point of $u$. 
\end{lem}

\begin{proof} 

Suppose that $p$ lies in the interior of a side $e$ of $p$.
Since $p$ is a local extremum of $u|_{\partial P}$, we have $L_e u(p)=0$. 
Thus, since $u$ satisfies Neumann conditions at $p$, we have 
$\nabla u=0$. 

\end{proof}


\section{Index zero critical points on a side of a polygon.}
\label{sec:zero-index}

In this section $u$ is a Neumann eigenfunction on a polygon $P$,
and $p$ is an isolated critical point of $u$ that lies 
in a side $e$ of $P$. 
We show in Lemma \ref{lem:tangential-cusp} that if $\ind(u,p)=0$, then 
the level set $\{z: u(z)=u(p)\}$ is a `cusp' that is tangent 
to $e$ (see Lemma \ref{lem:tangential-cusp}). 
We then use this to show that if the nodal set of $Xu$, where 
$X$ is either a rotational or constant vector field, has a degree 1 vertex,
then the vertex is a critical point with nonzero index
(Proposition \ref{prop:degree-1-stable}).

By applying a rigid motion to $P$ we may assume that $p=0$
and that the side that contains $p$ lies in the real axis.

\begin{lem} \label{lem:tangential-cusp}
Suppose that $p$ is an index zero critical point of 
a Neummann eigenfunction $u$ that belongs to the side $e$. 
Then there exist real-analytic functions $c: \Cbb \to \Rbb$ and $\rho:\Rbb \to \Rbb$ 
and an odd integer $k \geq 3$ so that $c(0) \neq 0$, $\rho(0)\neq 0$, and
\begin{equation} \label{eqn:tangential-cusp-equation} 
u(z)~
=~
u(0)~
+~
c(z) \cdot 
\left(  
y^2~
- x^k \cdot \rho(x)
\right).
\end{equation} 

\end{lem}

\begin{proof}
Because the index of the critical point $p$ of $u$ equals zero, 
the Hessian of $u$ has exactly one nonzero eigenvalue.
The eigenspace $E$ that corresponds to the nonzero eigenvalue is 
invariant under the reflection $z \mapsto \oz$. 
Thus $E$ is either the real or the imaginary axis.

We claim that $E$ is not the real axis. 
Indeed, suppose to the contrary that $E$ is the real axis. 
Then $\partial_x u(0)=0$ 
but $\partial_x^2\, u(0) \neq 0$. 
The Weierstrass preparation theorem
applies to provide unique real-analytic functions $a$, $b_1$, and $b_2$
defined near $0$, so that $a(0) \neq 0$, $b_1(0)=0=b_2(0)$, and 
\[
u(z)~ -~ u(0)~ =~ a(z) \cdot \left(x^2~ +~ b_1(y)\cdot x~ +~ b_2(y) \right)
\]
for $z$ near $p=0$. Since the factorization is unique and $u(\oz)=u(z)$,
we have $b_j(y)= b_j(-y)$ for $j=1,2$. In particular, the 
discriminant $D(y):= b_1(y)^2 - 4 \cdot b_2(y)$ is an even function. 
If $D$ were to vanish on a neighborhood of $0$, then 
we would have $u(z)- u(0) = a(z) \cdot ( x + b_1(y)/2)^2$ and hence
\[
\nabla u(z)~ 
=~ 
\left( x + \frac{b_1(y)}{2} \right)^2 \cdot \nabla a(z)~
+~
2 a(z) \cdot \left( x + \frac{b_1(y)}{2}\right) 
\cdot
\nabla \left( x + \frac{b_1(y)}{2}\right).
\]
Thus, $\nabla u$ would vanish along the level set of $u$ that contains $p=0$,
but by assumption $p=0$ is an isolated zero of $\nabla u$. 
Since $D$ is even and $\ind(u,0) \neq 1$, 
it follows that $D(y) >0$ for $y \neq 0$ sufficiently small, 
and hence there exists a neighborhood $U$ of $p=0$ so that 
the intersection of $u^{-1}(u(p))-\{p\}$ and $U$ 
consists of four arcs. This contradicts the assumption that $p$ is a 
zero index critical point of $u$.

Therefore $E$ coincides with the imaginary axis. By use of 
the Weierstrass preparation theorem we find that 
\begin{equation} \label{eqn:weierstrass-cusp}
u(z)~ -~ u(0)~ =~ a(z) \cdot \left(y^2~ +~ b_1(x)\cdot y~ +~ b_2(x) \right)
\end{equation}
for unique real-analytic functions $a$, $b_1$, and $b_2$
defined near $0$ where $a(0) \neq 0$ and $b_1(0)=0=b_2(0)$.
Since the factorization is unique and  $u(\oz)=u(z)$, we find that
$b_1(x)= 0$. Because $p$ is an isolated critical point, there exists 
$\epsilon>0$ so that $b_2(x) \neq 0$  if $0 < |x|< \epsilon$.
We claim that, moreover, $b_2(x) \cdot b_2(-x)<0$ if $0 < |x|< \epsilon$.
Indeed, otherwise $b_2(x) \cdot b_2(-x)>0$, 
and thus from (\ref{eqn:weierstrass-cusp}) we find that  
there exists a neighborhood $U$ of $p=0$ so that 
the intersection of $u^{-1}(u(p))-\{p\}$ and $U$ 
consists of four arcs. This contradicts the assumption that $p$ is a 
zero index critical point of $u$.

Since $b_2(x) \cdot b_2(-x)<0$ the first nonzero term in the Taylor 
series of $b_2$ about zero has odd degree $k$, and since $\partial_x u(0)=0$
we also have $k \geq 3$. The claim follows. 
\end{proof}

\begin{prop} \label{prop:degree-1-stable}
Let $X$ be either a constant vector field
or a rotational vector field. If $p$ is a degree 1 vertex of 
$\Zcal(X u)$ that is not a vertex of $P$, then $p$ is 
a critical point with nonzero index. 
\end{prop}

\begin{proof}
Since $p$ is not a vertex, $p$ lies in the interior 
of a side $e$. Since $p$ is a degree 1 vertex of $p$, the 
vector $X(p)$ is independent of the normal derivative at $p$,
and in particular $p$ is a critical point of $u$. 
It remains to show that $p$ has nonzero index.

As above, we may suppose without
loss of generality that $e$ lies in the real-axis and that $p=0$.
We will consider the case in which $X$ is a constant vector field
of the form $X= \cos(\psi) \partial_x + \sin(\psi) \partial_y$ 
where $\psi \neq \pi/2 \mod \pi$.  

Suppose to the contrary that the index of $p$ were to equal zero.
Then by Lemma \ref{lem:tangential-cusp}, near $p$, the function
$u$ would satisfy (\ref{eqn:tangential-cusp-equation}) 
where $c(0) \neq 0 \neq \rho(0)$ and $k \geq 3$ is odd. 
Direct computation 
shows that $\partial_y X u(0)= 2 \cos(\psi) \cdot c(0)$
and hence $\partial_y X u(0) \neq 0$.  Thus, by the implicit function
theorem, there exists a function $f:(- \epsilon, \epsilon) \to \Rbb$
so that 
\begin{equation} \label{eqn:implicit-arc}
Xu \left(x + i \cdot f(x) \right)~ =~ 0. 
\end{equation}
From (\ref{eqn:tangential-cusp-equation}), we find that for each real $x$
\begin{equation} \label{eqn:cusp-real-expansion}
 (Xu)(x)~ =~ -\cos(\psi) \cdot c(0) \cdot k \cdot x^{k-1}~ +~ O(|x|^k).
\end{equation}
By repeatedly differentiating (\ref{eqn:implicit-arc}) 
with respect to $x$ and using (\ref{eqn:cusp-real-expansion})
we find that $\partial_x^j f(0)=0$ for each $j<k-1$ and
$\partial_x^{k-1} f(0) \neq 0$. Since $k-1$ is even and greater than 0, 
the function $f$ is positive in a deleted neighborhood of $0$. Thus, 
there exists a neighborhood $U$ of $p=0$ such that 
$(U \cap \Zcal(X u))\setminus \{p\}$
lies in the upper half plane. Hence $p$ is not a degree 1 vertex,
a contradiction.
\end{proof}

The following Lemma follows from the discussion in \S 7 of \cite{J-M}.
We provide a statement and proof for the convenience of the reader.

\begin{lem} \label{lem:paralle-two-degree-1-vertices}
Let $u$ be a second Neumann eigenfunction on a polygon $P$,
and let $p$ be an index zero critical point that lies in the side $e$.
If $L_e$ is a constant vector field that is parallel to $e$, 
then  $\Zcal(L_e u)$ intersects the interior of $P$ 
and has at least two degree 1 vertices in $\partial P \setminus e$. 
\end{lem}

\begin{proof}
By Lemma \ref{lem:tangential-cusp}, $u$ has the expression 
$$
u(x, y) ~ =  u_{00} +  u_{02} \cdot y^2 + u_{30} \cdot (x^3 - 3x y^2) ~ + ~ O(4)
$$ 
in a neighborhood of $p$, where $u_{02} \neq 0$.
If $u_{30} = 0$ then Proposition 7.4 of \cite{J-M} provides the claim.
If $u_{30} \neq 0$ then the first paragraph of Proposition 7.6 of \cite{J-M} provides the claim.
\end{proof}

\section{Critical points of second Neumann eigenfunctions on simply connected polygons}
\label{sec:simply-connected}

In this section, we restrict attention to a polygon $P$ 
that is simply connected and to an eigenfunction $u$ 
that is associated to the second Neumann eigenvalue.

\begin{prop} \label{prop:simple-arc}
The nodal set $\Zcal(u)$ is a simple arc whose intersection
with $\partial P$ consists of its two endpoints. Moreover, the  
endpoints this arc lie in distinct side of $P$, and $\Zcal(u)$ does not contain any critical point of $u$.
\end{prop}

\begin{proof}
The first statement is a well-known consequence of 
Courant's nodal theorem and Polya's inequality.\footnote{See, for example, 
Theorem 5.2 \cite{J-M}.} 
The second statement follows from Lemma 3.3 \cite{J-M} and 
Theorem 2.5 in \cite{Cheng}.
\end{proof}

\begin{prop} 
\label{prop:index-at-least-minus-one}
Let $p$ be a critical point $p$ of a second 
Neumann eigenfunction $u$ that is not a vertex.
Then the index $\ind(u,p)$ equals either $1,0$, or $-1$.
\end{prop}

\begin{proof}
Let $\tu$ be the lift of $u$ to the double $DP$,
and let $\tp \in DP$ correspond to $p$.  If $\ind(u,p) < -1$,
then more than four arcs in $\tu^{-1}(\tu(\tp))$ 
emanate from $\tp$. 
It follows that, in the natural coordinates at $\tp$, 
we have $\tu(z)-\tu(\tp)= o(|z-\tp|^2)$.
In particular, the degree two homogeneous polynomial $h_2$ 
consisting  of second order terms in the Taylor expansion of $\tu$
at $p$ vanishes indentically. But $(\Delta h_2)(\tp) = \mu \cdot \tu(\tp)$
and so $u(\tp)=0$. 
This contradicts Proposition \ref{prop:simple-arc}.
\end{proof}

Next we consider the possible indices of a critical point 
of $u$ that lies at a vertex of $P$. Let $c_k$ be 
the coefficient in the Bessel expansion (\ref{eqn:Bessel-expansion-II})
at a point $p \in \partial M$.
The following should be compared with Corollary 5.3 in \cite{J-M}.

\begin{prop} 
\label{two:coeff:zero}
Let $p \in \partial P$. 
Either $c_0 \neq 0$ or $c_1 \neq 0$.
\end{prop}

\begin{proof}
If $c_0=0=c_1$, then by 
inspecting (\ref{eqn:Bessel-expansion-II})
one finds that at least two nodal arcs emanate from $p$.
This contradicts Proposition \ref{prop:simple-arc}.
\end{proof}

\begin{coro}  
\label{coro:1-connected-index}
Let $v$ be a vertex whose angle $\beta$ is not a multiple of $\pi/2$.

\begin{enumerate}
\item If $c_0 = 0$, then $\ind(u,v)=0$.  
\item If $\beta < \pi/2$, then $c_0 \neq 0$ if and only if $v$ is a local extremum.
\item If $\beta < \pi$ and $c_1=0$, then $v$ is a local extremum.

\item If $\pi/2<\beta < \pi$, then $c_1=0$ if and only if 
$v$ is a local extremum.
\item If $\beta< \pi$, then $\ind(u,v)=0$ or $\ind(u,v)=1$.

\end{enumerate}

\end{coro}

\begin{proof}

Let $k$ be the smallest positive integer such that $c_k \neq 0$.
If $c_0=0$, then by Proposition \ref{two:coeff:zero},
we have $c_1 \neq 0$, and hence by Lemma \ref{lem:vertex-index-1}
we have $\ind(u,v)=0$.

If $\beta < \pi/2$ and $c_0 \neq 0$, then we are in case (2)
of Lemma \ref{lem:vertex-index-1}, and hence $\ind(u,v)=1$. 
By Proposition 
\ref{prop:local-ext-index-1} we have $\ind(u,v)=1$ if and only if 
$v$ is a local extremum. 

Suppose that $\pi/2<\beta < \pi$. 
If $c_1 =0$, then Proposition \ref{two:coeff:zero} implies that $c_0 \neq 0$, 
and so part (ii) of Lemma \ref{lem:vertex-index-1} implies that $\ind(u,v)=1$.
If $c_1 \neq 0$, then part (i) of  Lemma \ref{lem:vertex-index-1}
implies that $\ind(u,v) = 0$.
\end{proof}

\begin{coro} \label{coro:c_0-vanish-L_eu-arc}
Suppose that $v$ is an acute vertex of $P$ contained in the side $e$. 
If $v$ is not a 
local extremum, then $\Zcal(L_{e} u)$ has an arc that ends at $v$.
\end{coro}

\begin{proof}
This follows from Corollary \ref{coro:arc-iff-leading-vanish} 
and Corollary \ref{coro:1-connected-index}. 
\end{proof}

\section{No hot spots on certain polygons with two acute vertices}

\label{sec:eigen-obtuse}

Ba\~{n}uelos and Burdzy \cite{B-B} used probabilistic methods
to show that the second Neumann eigenfunction $u$ of an obtuse triangle 
has no interior critical points. In \cite{J-M} \cite{Erratum}, we 
used a variational approach to show that the two acute vertices are 
the only critical points of $u$ and hence they are the global extrema of $u$. 
In this section, we extend this latter result to a large class of
$n$-gons that have two acute vertices. At the end of the section 
we identify this class of polygons as those that satisfy the Lip-1 
condition of \cite{A-B} and which have no orthogonal sides.

\begin{lem} \label{lem:lens-four-or-arc}
Let $u$ be a second Neumann eigenfunction on a simply connected
polygon $P$ with at least one acute vertex.
If $u$ has an interior critical point, then either 
$u$ has at least four nonzero index critical points
or there is a side $e$ such that $\Zcal(L_{e} u)$ 
has an arc that ends at a vertex of $P$. 
\end{lem}

\begin{proof}
Suppose that for each side $e$, the nodal set $\Zcal(L_e u)$ 
does not have an arc that ends at a vertex. 
Thus, if $v$ is a vertex and $e$ is a 
side containing $v$, then
Corollary \ref{coro:arc-iff-leading-vanish} implies
that the leading Bessel 
coefficient of $u$ at $v$ is nonzero.
In particular, each acute vertex has index $+1$
by Corollary \ref{coro:1-connected-index}, and each 
obtuse vertex is not a critical point by Proposition
\ref{coro:c_1-zero-nonzero-index}.

Since $u$ has a critical point in the interior of $P$,
for any side $e$,
the nodal set $\Zcal(L_e u)$ has at least two degree 1 
vertices in $\partial P$. Since $\Zcal(L_e u)$ 
does not have an arc that ends at a vertex of $P$, 
each of the degree one vertices of $\Zcal(L_e u)$ is a
non-vertex point on $\partial P$. By 
Proposition \ref{prop:degree-1-stable}, each of 
these degree 1 vertices is a nonzero index critical point of $u$.

Thus, since $P$ has at least one acute vertex, 
$u$ has at least three nonzero index critical points in $\partial P$. 
Since the obtuse vertices are not critical points, 
Proposition \ref{prop:index-at-least-minus-one} implies 
that each critical point of $u$ that lies in $\partial P$ 
has index equal to $1$, $0$ or $-1$. Thus, it follows 
from Proposition \ref{prop:index-thm} that the 
number of nonzero index critical points of $u$ that lies on $\partial P$
is even. In particular, $u$ has at least four nonzero index 
critical points on $\partial P$.
\end{proof}

In the following we consider paths $P_t$ of polygons with $n$ vertices.
We do not allow for two vertices to collide.
But we do allow for the angle of a vertex to become equal to $\pi$. 
For example, a triangle $T$ may be regarded as a quadrilateral
if we declare that a point on some side of $T$ is a vertex 
with angle $\pi$.

Let $u_t$ be a path of second Neumann eigenfunctions associated to 
the path $P_t$. For each $t$, let $V(t)$ denote the number of vertices $v$
of $P_t$ with angle not equal to $\pi$ such that there exists a side 
$e$ of $P_t$ so that an arc in $\Zcal(L_e u_t)$ ends at $v$.
Let $S(t)$ denote the number of nonzero index critical points of $u_t$.

\begin{lem} \label{lem:NUT-open-closed}
Suppose that $P_t$ is a path of $n$-gons such that 
no two sides of $P_t$ are orthogonal and $P_t$ has 
exactly two acute vertices for each $t \in [0,1]$. 
Let $u_t$ be an associated path of eigenfunctions. 
If $S(0) \geq 3$ or $V(0) \geq 1$, then either $S(1) \geq 3$
or $V(1) \geq 1$.
\end{lem}

\begin{proof}
It suffices to show that the set, $A$, of $t \in [0,1)$ such that either 
$S(t) \geq 3$ or $V(t) \geq 1$ is both open and closed in $[0,1)$. 

\vspace{.2cm}

($A$ is open) If  $S(t) \geq 3$, then Lemma \ref{lem:top-stable-implies-stable}
implies that there exists $\epsilon>0$ such that if $|s-t|< \epsilon$,
then $S(s) \geq 3$. 
Hence, to prove openness, it suffices to assume that $V(t) \geq 1$, and  show that there exists $\epsilon>0$ so that if 
$|s-t|< \epsilon$ then either $V(s) \geq 1$ or $S(s) \geq 3$.

If $V(t) \geq 1$, then there exists vertex $v$ of $P_t$, 
a side $e$ of $P_t$, and an arc in $\Zcal(L_e u)$ that ends at  $v$. 

If the leading coefficient at $v$ is nonzero then 
for $s$ near $t$ the leading coefficient at the corresponding vertex 
is also nonzero. Since no two sides of $P_t$ are orthogonal 
and since the corresponding edge and sector
vary continuously in $s$, we find from Lemma \ref{lem:nodal-sector}
that $V(s) \geq 1$ for each $s$ near $t$. 
Therefore, we may assume that there exists a vertex $v$ such that  the leading coefficient of $u_t$ at $v$ equals zero.

We may assume without loss of generality that $v$ is an acute vertex. 
Indeed,  if $v$  were obtuse with $c_1(t)=0$, 
then by Lemma \ref{lem:vertex-index-1} $v$ would be a critical 
point with nonzero index. If $c_0(t)$ were not to vanish at each of the two acute 
vertices then Lemma \ref{lem:vertex-index-1} would imply that
each of these vertices have index equal to one.

Hence, $S(t) \geq 3$,
and so $S(s) \geq 3$ for $s$ near $t$ by Lemma
\ref{lem:top-stable-implies-stable}.
Thus, we may assume that $v$ is acute.

Suppose that $c_0(t)=0$ at an acute vertex $v$. Corollary 
\ref{coro:1-connected-index} implies that $\ind(u,v)=0$.
By Proposition \ref{lem:two-stable-on-boundary},
the eigenfunction $u_t$ has at least two nonzero index critical 
points. Thus, it follows from Lemma \ref{lem:top-stable-implies-stable}
that there exists $\epsilon>0$ such that if $|s-t|<\epsilon$,
then there exist two nonzero index critical points 
of $u_s$ that are distinct from $v$.
Suppose that $0 <|s-t| < \epsilon$. If $c_0(s) \neq 0$, 
then, since $v$ is acute,  
Corollary \ref{coro:1-connected-index} implies that 
$\ind(u_s,v) \neq 0$, and hence $S(s) \geq 3$. 
On the other hand, if $c_0(s) = 0$,
then Corollary \ref{coro:c_0-vanish-L_eu-arc} implies that $\Zcal(L_e u_s)$ 
has an arc 
that ends at $v$ where $e$ is a side adjacent to $v$.
In sum, if $|s-t|< \epsilon$, then either $S(s) \geq 3$ or  
$V(s) \geq 1$.

\vspace{.2cm}

($A$ is closed) By assumption, for each $t \in [0,1)$, no two sides of 
$P_t$ are orthogonal, and so the set of $t$ such that $V(t) \geq 1$ is 
closed by Lemma \ref{lem:nodal-sector}. Suppose that $S(t_n) \geq 3$ 
with $t_n \to t$. To prove that $A$ is closed
it suffices to show that either $S(t) \geq 3$ or $V(t) \geq 1$. 

If the eigenfunction $u_t$ has an interior critical point,
then Lemma \ref{lem:lens-four-or-arc} implies that $S(t) \geq 4$
or $V(t) \geq 1$. 
Thus, we may assume that $u_t$ has no interior critical points.
By Proposition \ref{lem:two-stable-on-boundary}, 
the two index 1 critical points, $p^+$ and $p^-$ lie in $\partial P$.
Suppose that there exists a third critical point $p$.
If the index of $p$ is nonzero, then $S(t)\geq 3$. 
Thus, in the following we assume that $\ind(u,p) = 0$.

If some acute vertex $v$ has is not a local extremum, then by 
 Corollary \ref{coro:c_0-vanish-L_eu-arc} an arc of $\Zcal(L_e v)$  ends at $v$,
and so $V(t) \geq 1$. Thus, we may assume that each acute 
vertex is a local extremum. If there are three local extrema,
then $S(t)\geq 3$.
Hence we may assume that the 
acute vertices correspond to the the index 1 critical points $p^+$ and $p^-$. 

Because $S(t_n) \geq 3$, for each $n$ there exists a critical point $p_n$ on a side that is distinct from $p^+$ and $p^-$. Suppose that $p_n$ converges to a 
vertex $v$ of $P_t$ whose angle does not equal $\pi$. 
Then, by Lemma \ref{lem:convergence-to-vertex-coefficient-to-zero}
the leading coefficient---$c_0(t)$ if $v$ is 
acute and $c_1(t)$ if $v$ is obtuse---equals zero. 
If $v$ is obtuse then Corollary \ref{coro:c_1-zero-nonzero-index} implies that $v$ is a 
nonzero index  critical point
and so $S(t) \geq 3$. If $v$ is acute, then by  
Corollary \ref{coro:c_0-vanish-L_eu-arc} we have $V(t) \geq 1$.   

Therefore, we may assume that $p_n$ converges to a critical point $p$ 
of $u_t$ that lies in the interior of a side $e$. Since $p \neq p^{\pm}$,
the critical point has index equal to zero. Thus, by Lemma 
\ref{lem:paralle-two-degree-1-vertices}, the graph $\Zcal(L_e u_t)$ 
intersects the interior of $P_t$ and has at least two degree 1 vertices. 
If one of these degree $1$ vertices equals a vertex of $P_t$ then 
$V(t) \ge 1$. If a degree one vertex lies in the interior of a side
then it is a nonzero index  critical point by Proposition
\ref{prop:degree-1-stable}, and hence $S(t) \geq 3$ since the acute 
vertices $p^{\pm}$ are also nonzero index  critical points. 
\end{proof}

\begin{thm} 
\label{thm:only-acute}
Suppose that $P_t$ is a path of $n$-gons such that 
no two sides of $P_t$ are orthogonal. 
If $P_1$ is an obtuse triangle, then each 
second Neumann eigenfunction of $P_0$ has exactly two
critical points, a global maximum 
at one acute vertex and a global minimum at the other acute
vertex. Moreover, the second Neumann eigenspace of $P_0$ 
is one-dimensional.
\end{thm}

\begin{proof}
By  the method of
Lemma 12.2 of \cite{J-M}, one may modify the path $P_t$ 
so that there exists a continuous family of second Neumann 
eigenfunctions  $u_t$ connecting any $u_0$ to any $u_1$.
If $u_1$ is a second Neumann eigenfunction for an obtuse triangle $P_1$, 
then by \cite{J-M} \cite{Erratum}, the acute vertices are the only 
critical points of $u_1$, and in particular each is a global extremum.
Thus Proposition \ref{prop:local-ext-index-1} and 
Corollary \ref{coro:1-connected-index} imply that 
the coefficient $c_0$ of $u_1$ at each acute vertex
is nonzero. Given an acute vertex $v$, the angle between 
opposite side and one of the sides adjacent to $v$ is greater 
than $\pi/2$. Hence it follows from Lemma \ref{lem:nodal-sector} that
for each side $e$ of $P_1$
there does not exist an arc in $\Zcal(L_e u)$ that ends at 
an acute vertex. The obtuse vertex is not a local extremum
and hence $c_1$ of $u_1$ at this vertex is nonzero. Thus, it follows from 
Lemma \ref{lem:nodal-sector} that for each side $e$ of $P_1$, 
no arc of $\Zcal(L_e u_1)$ ends at the obtuse vertex. 
In sum, $S(1) =2$  and $V(1)=0$.

Thus,  Lemma \ref{lem:NUT-open-closed} implies that
$S(0) = 2$ and $V(0)=0$.  In particular,  $u_0$ has exactly 
two nonzero index critical points and these 
are necessarily the global extrema of $u_0$. 
Each global extremum must be an acute vertex.
Indeed if an acute vertex $v$ of $P$ were not a local extremum, 
then by Corollary \ref{coro:c_0-vanish-L_eu-arc}
we would have that $\Zcal(L_e v)$ has an arc that ends at $v$
where $e$ is a side adjacent to $v$, contradicting $V(0)=0$.

Suppose that there exists a critical point $p$ of $u_0$ 
that were distinct from the acute vertices. Then $p$  has index zero 
and lies in a side of $P_0$.  
Thus, $p$ lies in the interior of a side $e$ of $P$, and hence by Lemma
\ref{lem:paralle-two-degree-1-vertices}, the graph $\Zcal(L_e u_t)$ 
intersects the interior of $P_t$ and has at least two degree 1 vertices.
If a degree 1 vertex $p$ lies in the interior of a side, 
then $\ind(u,p)\neq 0$ by Proposition \ref{prop:degree-1-stable},
a contradiction. Therefore, the acute vertices are the only
critical points of $u_0$.

Finally, we show that the second Neumann eigenspace of $P_0$
is one-dimensional. Let $u_+$ and $u_-$ be second Neumann eigenfunctions 
of $P_0$ and let $v$
be an acute vertex. Then there exist $a_+, a_- \in \Rbb$ so that 
$a_+ \cdot u_+(v) + a_- \cdot u_-(v)=0$. We claim that 
$u^*:=a_+ \cdot u_+ + a_- \cdot u_- \equiv 0$. Indeed, if not then 
$u^*$ would be a second Neumann eigenfunction and in particular 
would be orthogonal to the constant functions. 
Thus both the the maximum value and the minimum value of $u$
would be nonzero. But by Theorem \ref{thm:only-acute}, the acute vertex
$v$ is a global extremum of $u^*$ and hence we have a contradiction. 
\end{proof}

We now show that the set of polygons that satisfy 
the hypotheses of Theorem \ref{thm:only-acute}
is the interior of the set of polygons that satisfy
the Lip-1 condition of \cite{A-B}. First we recall,
the notion of Lip-K domain. 
Let $f_{+}:[-b,b] \to \Rbb$ and $f_-:[-b,b] \to \Rbb$ 
be a pair of Lipschitz functions such that 
\begin{itemize}

\item $f_+(\pm b)= f_{-}(\pm b)$, 

\item $f_-(x) < f_+(x)$ for $x \in (-b,b)$, and

\item  the Lipschitz constant of $f_{\pm}$ is at most $K$. 
\end{itemize}
The domain 
$\{(x,y)\, :\,  f_-(x) < y < f_+(x)\}$ is called a {\em Lip-$K$ domain}.

Recall that if $\Omega$ is a domain with Lipschitz boundary 
$\partial \Omega$ then the outward unit normal vector $\nu(p)$
is defined for almost every $p \in \partial \Omega$. 

\begin{prop}
\label{prop:lip}
A simply connected Lipschitz domain $\Omega$ 
is isometric to a Lip-1 domain if and only if 
there exists a partition of $\partial \Omega$ into 
two connected sets $\Gamma^+$ and $\Gamma^-$ 
so that if $p, p' \in \Gamma^{\pm}$ then $\nu(p) \cdot \nu(p') \geq 0$
and if $p \in \Gamma^+$ and $q \in \Gamma^-$ then 
$\nu(p) \cdot \nu(q) \leq 0$.
\end{prop}

\begin{proof}
$(\Rightarrow)$
After applying an isometry, we may suppose that 
$\Omega$  is bounded by the graphs of the Lip-1 functions
$f_+$ and $f_-$ as above. Let $\Gamma^+$ be the graph of $f^+$
and let $\Gamma^-$ be the graph of $f^-$. Suppose that 
$\nu(p)= (x,y)$.  Since $f^+$ is Lip-1 we have that
$p \in \Gamma^+$ implies that $y > |x|$, and since $f^-$ 
is Lip-1 we have that  $p \in \Gamma^-$ implies that $y < -|x|$. 
It follows that if $p, p' \in \Gamma^{\pm}$ then 
$\nu(p) \cdot \nu(p') \geq 0$
and if $p \in \Gamma^+$ and $q \in \Gamma^-$ then 
$\nu(p) \cdot \nu(q) \leq 0$.

$(\Leftarrow)$
Let $p_n^+ \in \Gamma^+$ and $p_n^- \in \Gamma^-$ be sequences
such that $\lim_{n \to \infty} \nu(p_n^+) \cdot \nu(p_n^-)$
equals the supremum of 
$\{ \nu(p) \cdot \nu(q)\, :\, p \in \Gamma^+,\, q \in \Gamma^-\}$.
Let $w$ be a limit point of the sequence 
$\left(\nu(p_n^+) - \nu(p_n^{-} \right)/|\nu(p_n^+) - \nu(p_n^{-})|$.
A computation shows that for each $p \in \Gamma^+$
we have $\nu(p) \cdot w \geq 1/\sqrt{2}$ and for each 
$p \in \Gamma^-$ we have $\nu(p) \cdot w \leq -1/\sqrt{2}$.
Choose coordinates in the plane so that the vector $w$ is the vector $(0,1)$.
Then for each $p \in \Gamma$ we have $\nu(p)= (x,y)$ where $y \geq |x|$.
From this it follows that $\Gamma^+$ is the graph of a Lip-1 function 
$f_+:[a_+,b_+] \to \Rbb$.
Similarly, $\Gamma^-$ is the graph of a Lip-1 function 
$f_-:[a_-,b_-] \to \Rbb$..
Because $\Gamma^+$ and $\Gamma^-$ form a partition of $\partial \Omega$
we have $f_{+}(a_+)= f_-(a_-)$ and $f_{+}(b_+)= f_-(b_-)$.
Because $\nu(p)$ is the outward normal vector for a domain 
we have $f_+> f_-$.
\end{proof}

\begin{coro}
\label{coro:obtuse-Lip}
A triangle $T$ is a Lip-1 domain if and only if $T$ is not an acute triangle.
\end{coro}

\begin{proof}
Let $e_1$, $e_2$, $e_3$ be the sides of the 
triangle and let $\nu_1$, $\nu_2$ and $\nu_3$ be the associated 
outward normal vectors. The angle between $e_i$ and $e_j$ is acute
if and only if $\nu_i \cdot \nu_j <0$. The claim follows from 
Proposition \ref{prop:lip}.
\end{proof}

\begin{prop}
Suppose that $P_t$ is a path of polygons such that no two
sides of $P_t$ are orthogonal and $P_0$ is isometric to a Lip-1 domain. 
Then each $P_t$ is also isometric to a Lip-1 domain. 
\end{prop}

Note that we are allowing for the possibility that some vertices have angle $\pi$ for some $t$.

\begin{proof}
Since $P_0$ is a Lip-1 domain, there exists a partition  
$\{\Gamma^+,\, \Gamma^-\}$ of $\partial P_0$ that satisfies the 
criteria of Proposition \ref{prop:lip}.  In particular, 
$\Gamma_{+}$ is the union of sides 
with outward unit normal vectors $\nu_1^{+}(0), \ldots, \nu_j^{+}(0)$,
the set $\Gamma_{-}$ is the union of sides 
with outward unit normal vectors $\nu_1^{-}(0), \ldots, \nu_k^{-}(0)$,
and these normal vectors satisfy $\nu_i^{\pm}(0) \cdot \nu_j^{\pm}(0) \geq 0$ 
and $\nu_i^+(0) \cdot \nu_j^-(0) \leq 0$. 
Since no two sides of $P_0$ are orthogonal, each inequality is strict.
The quantities  $\nu_i^{\pm}(t) \cdot \nu_j^{\pm}(t)$ and 
$\nu_i^+(t) \cdot \nu_j^-(t)$ depend continuously in $t$ and
cannot vanish since no two sides of $P_t$ are orthogonal. 
Thus the inequalities persist for all $t$, and thus each $P_t$
is a Lip-1 domain by Proposition \ref{prop:lip}. 
\end{proof}

\begin{prop}
\label{prop:lip-admissible}
If $P$ is a Lip-1 polygonal domain with no two sides orthogonal, 
then there exists a path $P_t$ of polygons with no two sides 
orthogonal such that $P_1=P$ and $P_0$ is an obtuse triangle.  
\end{prop}

\begin{proof}
We will argue via induction on the number, $n$, of sides of $P$.
If $n=3$, then the claim follows from \ref{coro:obtuse-Lip}.
Suppose that the claim is true if a Lip-1 polygon has $n$ sides
no two of which are orthogonal. Let $P$ be a Lip-1 polygon
with $n+1$ sides such that no two sides are othogonal.
Proposition \ref{prop:lip} implies that the sides of $P$ 
can be partitioned into sides  $e_1^+, \ldots, e_j^+$
and $e_1^-, \ldots, e_k^-$, so that the associated 
outward unit normal vectors $\nu_1^+, \ldots, \nu_j^+$ 
and $\nu_1^-, \ldots, \nu_k^-$ satisfy the inequalities 
$\nu_i^{\pm} \cdot \nu_j^{\pm} > 0$ and $\nu_i^+ \cdot \nu_j^- < 0$.
Because $P$ has nonempty interior, by relabeling if necessary,
we may assume that $\nu_1^+ \neq \nu_2^+$ and the sides 
$e_1^+$ and $e_2^+$ are adjacent. Let $v$ be the vertex 
shared by $e_1^+$ and $e_2^+$, and let $v'$ be the midpoint
of the segment that joins the other two vertices of 
$e_1^+$ and $e_2^+$. Define $P_t$ to be the polygon obtained from $P$ 
by replacing $v$ with $v_t = (1-t) \cdot v + t \cdot v'$.
A straightforward computation show that both $n_1^+(t)$
and $n_2^+(t)$ are convex combinations of $n_1^+$ and $n_2^+$,
and so it follows that $P_t$ is a Lip 1-polygon with no
orthogonal sides. The polygon $P_1$ may be regarded as
a Lip-1 polygon with only $n$ sides no two of which are orthogonal.
Thus, by the inductive hypothesis, 
we may concatenate the path $P_t$ with another path
to obtain the desired path to an obtuse triangle.
\end{proof}


\section{Instability via blocking}

\label{sec:blocking}

In this section we provide criteria---Proposition \ref{prop:blocking}---that
guarantee the existence of a quadrilateral with a second Neumann 
eigenfunction that has an unstable critical point.
In \S \ref{sec:break}, we will construct families of quadrilaterals
that meet the criteria under the assumption that these quadrilaterals
have no interior critical points. 

The statement and proof of Proposition \ref{prop:blocking} are
somewhat complicated, but the basic idea is simple:
Suppose that we have a continuous family of quadrilaterals 
$Q_t$ with an obtuse vertex $w_t$ and sides $ e^-_t$ and $e^+_t$
 adjacent to $w_t$. 
Suppose further that for the associated family of eigenfunctions
$u_t$, we know that $u_0$ (resp. $u_1$) has 
only one nonvertex critical point $p_0$ (resp. $p_1$),
that this critical point lies
on the side $e^-_0$ (resp. $e^+_1$), 
and that this critical point has index $-1$. 
One might naively expect that the index $-1$ critical point 
varies continuously in $t$, and therefore, for some time $t$,
the critical point lies at the obtuse vertex $w_t$. 
However, Lemma \ref{lem:interior-crit-pt-convergence} would 
then imply that $c_1=0$ at $w_t$, and then  Corollary
\ref{coro:1-connected-index} would imply that $w_t$ is an
index $+1$ critical point. Thus, the index of the critical 
point would abruptly change which is not possible by 
Proposition \ref{prop:local-number-of-critical-points}.
Roughly speaking, the obtuse vertex `blocks' the 
index $-1$ critical point. 

Under additional assumptions, we show that this 
`blocking phenomenon' implies the existence of 
an unstable critical point.

\begin{prop}
\label{prop:blocking}
Let $Q_t$ be a continuous family of quadrilaterals such that for each $t \in [0,1]$
the quadrilateral $Q_t$ has three acute vertices, and the angle of the 
fourth vertex, $w_t$, lies in $(\pi/2, \pi)$ for each $t \in (0,1)$. 
Let $e_t$ be a side of $Q_t$ that is adjacent 
to $w_t$ so that $t \mapsto e_t$ is continuous.
Let $u_t: Q_t \to \Rbb$ be a second Neumann eigenfunction, 
and suppose that $t \mapsto  u_t$ is continuous. Suppose that 
\begin{enumerate}

\item for each $t$, the eigenfunction $u_t$ has no interior critical points,

\item for each $t$, each nonzero index critical point 
      of $u_t$ either is a vertex or belongs to 
      a side adjacent to $w_t$,
     
\item for each $t$, each acute vertex of $Q_t$ is a local extremum of $u_t$,

\item $u_0$ has exactly one nonvertex critical point and it
      belongs to the interior of $e_0$.

\item $u_1$ has no critical points on $e_1$ except for the acute vertex.
\end{enumerate}
Then there exists $t \in (0,1)$ such that $u_t$ has an unstable critical point.
\end{prop}

\begin{proof}
For each $t \in [0,1]$, let $A_t$ be the set of critical points $p$ of $u_t$
such that either $p=w_t$ or $p$ lies in the interior of a side of $Q_t$
that is adjacent to $w_t$. 
We claim that there exists $\delta>0$ so that for all $t$
no element of $A_t$
is within distance $\delta$ of an acute vertex. Indeed, if not, then 
there would exist $t \in [0,1]$, a sequence $t_n \to t$,
and a sequence of critical points $p_n$ of $u_n$
that converges to an acute vertex $v$. 
Lemma \ref{lem:interior-crit-pt-convergence} would then imply that  
$c_0 = 0$ at $v$, but this would contradict part (3) of Corollary
\ref{coro:1-connected-index} and condition (3) above. 

By condition (4), the set $A_0$ has exactly one element $p_0$,
and it follows from Proposition \ref{prop:index-thm}
that the index of $p_0$ equals $-1$. 
Thus, Proposition \ref{prop:local-number-of-critical-points} 
implies that the sum of the indices of the critical points 
in $A_t$ equals $-1$.

Let $t^*$ be the supremum of $t \in [0,1]$ 
such that $A_s$ contains exactly one nonzero index point, $p_s$, 
for each $s \leq t$. It follows from Proposition \ref{prop:local-number-of-critical-points}
that $s \mapsto p_s$ is continuous on $[0,t^*)$
and the index of each $p_s$ equals $-1$. Moreover, as $s \nearrow t^*$
the point $p_s$ converges to a point $p_{t^*}$.  
Indeed, if $p_t$ were to have more than one limit point as $t \nearrow t^*$, 
then, since $t \mapsto p_t$ is continuous 
for $t< t^*$, we would have a nontrivial 
continuum of critical points. But since $Q_{t^*}$ is not 
a rectangle,
the function $u_{t^*}$ has only finitely many critical 
points \cite{J-M-arc}. 

Proposition \ref{prop:local-number-of-critical-points}
implies that the index of $p_{t^*}$ equals $-1$.
It follows that
the critical point $p_{t^*}$ lies in the interior of $e_{t^*}$.
Indeed, otherwise, condition (5) would imply that $p_s =w_s$
for some $s \leq t^*$. But this would contradict part (5) 
of Corollary \ref{coro:1-connected-index}.

 If $A_{t^*}$ contains a critical point $q$ that is distinct 
from $p_{t^*}$ then $q$ is an unstable critical point since $p_s$
is the only critical point in $A_s$ for $s<t^*$.
For the remainder of the proof we will suppose 
that $p_{t^*}$ is the only element of $A_{t^*}$.

By the definition of $t^*$, there exists a sequence 
$t_n \searrow t^*$ such that $A_{t_n}$ consists of more
than one nonzero index critical point. 
Since $p_{t^*}$ is the only critical point in $A_{t^*}$
these points converge to $p_{t^*}$,
and in particular for $n$ sufficiently large, the set $A_{t_n}$ lies 
in the interior of $e_{t_n}$. 
By Proposition \ref{prop:index-at-least-minus-one}, each
nonzero index critical point has index $+1$ or $-1$. Hence 
since the sum of the indices equals $-1$, the set $A_{t_n}$ 
contains at least three critical points.

But this is impossible. 
Indeed, if three critical points of $u_{t_n}$ 
were to lie in the interior of $e_{t_n}$, 
then $\Zcal(L_{e_{t_n}}  u_{t_n})$ would have three degree 1 vertices 
that lie in $\partial Q_{t_n} \setminus e_{t_n}$, 
and in particular some degree 1 vertex
would lie in the interior of a side not adjacent to $w_{t_n}$.
But this degree 1 vertex would be a nonzero index 
critical point by Proposition \ref{prop:degree-1-stable}, 
thus contradicting (2). 

\end{proof}


\section{Breaking acute triangles along a side}
\label{sec:break}

In this section, we will construct families of quadrilaterals
that satisfy the hypotheses (2) through (5) of Proposition \ref{prop:blocking}. The construction consists of:

\begin{enumerate}[1.]
\item Producing a nonempty open set $\Ncal$ of acute triangles 
    $T$ such that the set of critical points of each 
    second Neumann eigenfunction $u$ consists only 
    of the three vertices and an index $-1$ 
    critical point;
    
\item `Breaking' the side that contains the index
   $-1$ critical point of $T \in \Ncal$ to create
   quadrilaterals for which each acute vertex is a critical point and  
   for which the only sides that may contain critical points in their interior
   are the sides adjacent to the new obtuse vertex;
     
\item Choosing a path $w_t$ of break points 
    so that the resulting path $Q_t$ of quadrilaterals
    forces `blocking' to occur.

\end{enumerate}

We now provide the details of this construction.
Define $\Ncal$ to be the set of acute triangles $T$ such that
if $u$ is any second Neumann eigenfunction on $T$, then 
\begin{enumerate}
\item each vertex of $T$ is a local extremum of $u$,

\item $u$ has exactly one nonvertex critical point $p$.

\item the critical point $p$ is nondegenerate.
\end{enumerate}
Proposition \ref{prop:index-thm} implies that $p$ has index equal to $-1$. 
The main theorem of \cite{Erratum} implies that $p$ lies on a side of $T$.
Note that equilateral triangles do not belong to $\Ncal$,
and hence by a result of Siudeja \cite{Siudeja}, the second Neumann 
eigenspace of $T$ is one dimensional for each $T \in \Ncal$.

\begin{lem}
\label{lem:N-open}
The set $\Ncal$ is open in the space of acute triangles.
\end{lem}

\begin{proof}
By Corollary \ref{coro:1-connected-index}, a vertex $v$ 
of an acute triangle is a local extremum if and only if $u(v) \neq 0$.
Thus, condition (1) is open. 
A critical point $p$ is nondegenerate if and only if the
determinant of the Hessian at $p$ is nonzero, and hence 
condition (3) is also an open condition.

Thus, if $\Ncal$ were not open, then there would exist $T \in \Ncal$ and 
a sequence $T_n$ converging to $T$ such that condition (2) is not satisfied 
for each $n$. In particular, for each 
$n$ there would exist a second Neumann eigenfunction $u_n$ on $T_n$ 
with distinct nonvertex critical points $p_n$ and $q_n$. 

By passing to a subsequence if 
necessary, we may assume without loss of generality that $u_n$ 
converges to an eigenfunction $u$ on $T$.
Neither of the sequences $p_n$ nor $q_n$ can converge to 
a vertex of $T$ because then, by 
Lemma \ref{lem:interior-crit-pt-convergence}, we would have $c_1=0$
contradicting (1). Thus, by (2), both sequences converge to the unique 
nonvertex critical point $p$ of $T$, and it would
follow that $p$ is a degenerate critical point, contradicting (3). 
\end{proof}

The set $\Ncal$ is also nonempty.

\begin{lem} \label{lem:isosceles}
\label{lem:subequilateral}
Let $T$ be an isosceles triangle with reflection symmetry $\sigma$,
and let $u$ be a second Neumann eigenfunction of $T$.
If the angle of the apex vertex $v$ fixed by $\sigma$ is less than $\pi/3$, 
then 
\begin{enumerate}
\item each vertex is a local extremum of $u$,

\item $u$ has exactly one non-vertex critical point $p$,
      the midpoint of the side $e$ opposite to $v$,

\item $p$ is nondegenerate with  index $-1$,
     
\item  $u(z) \neq 0$ for each $z \in e$. 
\end{enumerate}
\end{lem}

\begin{proof}
By Lemma 3.1 in \cite{Miyamoto} the second Neumann eigenvalue of $T$ has 
multiplicity one, and $u$ is symmetric with respect to $\sigma$. 
It follows that the the midpoint $p$ of the side $e$ preserved by $\sigma$
is a critical point.
Let $T_+, T_- \subset T$ be the two right triangles such that $\sigma(T_+)=T_-$
and $T_+ \cup T_-= T$. Since $u$ is symmetric with respect to $\sigma$, 
the restriction of $u$ to $T_{\pm}$ is a second Neumann 
eigenfunction of $T_{\pm}$.

By Theorem 4.1 in \cite{Erratum}, the restriction 
of $u$ to the right triangle $T_{\pm}$ has no nonvertex critical 
points and each acute vertex of $T_{\pm}$ is a local extremum.
It follows that each vertex of $T$ is a local extremum of $u$ and the midpoint $p$
of $e$ is the only other critical point of $u$. 
Thus, Theorem \ref{prop:index-thm} implies that $p$ has index $-1$.

Next we show that $u$ does not vanish on $e$.
By Proposition \ref{prop:simple-arc}, the nodal set $\Zcal(u)$ 
does not contain a critical point
and hence does not contain the midpoint $p$. 
Thus, if there did exist $z \in e$ with $u(q)=0$, 
then $z \neq p$ and hence $\sigma(z) \neq q$. Since $u$ is
symmetric, we would have $u(\sigma(z))=0$ but then $\sigma(z)$ 
would be a
second endpoint of $\Zcal(u)$ that lies in $e$, a contradiction.
Therefore, $u$ does not vanish on $e$.

Finally, by examining the Taylor expansion of $u$ about $p$,
we find that the $p$ is non-degenerate. Indeed,
without loss of generality,
$q=0$ and $e$ lies in the $x$-axis. Since $u \circ \sigma =u$,
the restriction of $u$ to $e$ is an even function of $x$.  
In particular, the Taylor coefficient $a_{30}=0$. Thus, if
$p$ were degenerate, then Theorems 7.3 and 7.4 in \cite{J-M}
would imply that $u$ has an additional non-vertex critical point,
a contradiction.
\end{proof}

Next, we will `break' each $T \in \Ncal$ along the side that contains
the index $-1$ critical point. We first give a precise definition
of `breaking': Let $T$ be a triangle\footnote{One can easily extend the notion of 
breaking along a side to general polygons.} with vertices $v_1, v_2, v_3$. 
Let $e$ be a side of $T$,  let $w$ be a point that lies in the interior of $e$, 
and let $n_w$ be the outward pointing unit normal 
vector at $w$. For each $\epsilon \geq 0$, define 
$w(\epsilon)= w + \epsilon \cdot n_w$, and define 
$Q(T, w, \epsilon)$ to be the convex hull of 
$\{v_1, v_2, v_3, w(\epsilon) \}$. For $\epsilon>0$, the polygon 
$Q(T, w, \epsilon)$ is a nondegenerate quadrilateral. 
We say that $Q(T, w, \epsilon)$ 
{\it is the result of breaking $T$ along $e$ at the point $w$
at distance $\epsilon$.}

\begin{lem}
\label{lem:perturb-quad}
Let $T \in \Ncal$ and let $e$ be the side of $T$ that 
contains the index $-1$ critical point. Let $K$ be a 
compact subset of the interior of $e$.
There exists $\delta>0$ such that 
if $0 \leq \epsilon < \delta$ and $w \in K$, then
\begin{enumerate}[(a)]
    \item the second Neumann eigenfunction $u$ of $Q(T, w, \epsilon)$ 
          is unique up to scalar multiplication,

    \item each acute vertex of $Q(T, w, \epsilon)$ is a local extremum of $u$,
    
    \item if $e'$ is a side that does not contain the obtuse vertex
          $w$, then the interior of $e'$ does 
          not contain a critical point of $u$.  
\end{enumerate}
\end{lem}
\begin{proof}
The simplicity of the second Neumann eigenvalue is an open condition, 
and the second eigenvalue of each $T \in \Ncal$ is simple by \cite{Siudeja}.
It follows that there exists $\delta'>0$, so that (a) holds
for each $Q(T,w, \epsilon)$ with $\epsilon <\delta'$ and $w \in K$.
Corollary \ref{coro:1-connected-index} implies that 
condition (b) is an open condition. In particular,
$c_0 \neq 0$ at each acute vertex. 

Thus, if the claim were false, then there 
would exist a sequence $\epsilon_n \to 0$ and $w_n \in K$ 
such that $Q_n:=Q(T, w_n,\epsilon_n)$ 
has a second Neumann eigenfunction $u_n$ with 
a nonvertex critical point $p_n$ on a side $e'$
that does not contain $w_n(\epsilon_n)$. The sequence $Q_n$ 
converges to $T$, and thus by passing to a subsequence if necessary, 
we may assume that $u_n$ converges to an eigenfunction $u$ on $T$.
If the sequence $p_n \in e'$ were to converge 
to a vertex $v$ of $T$, then Lemma \ref{lem:interior-crit-pt-convergence}
would imply $c_1=0$ at $v$, a contradiction. 
If the sequence $p_n$ converges to a point $p$
in the interior of $e'$, then $p$ is a critical
point of $u$, contradicting the assumption that
the `unbroken' sides of $T$ contain no critical points.
\end{proof}

Let $\delta_{T,K}$ denote the supremum of all possible $\delta$ 
for which the statement of Proposition \ref{lem:perturb-quad}
is true for the given compact set $K$.

\begin{lem}
\label{lem:blocking-fam}
Let $T \in \Ncal$ and let $e$ be the side of $T$ that contains 
the index $-1$ critical point $p$. Let $w_t$ be a path in the 
interior of $e$ so that $w_0$ and $w_1$ lie in distinct components of 
$e \setminus \{p\}$. If $K$ is the image of the path $w_t$,
then for each $\epsilon \in (0,\delta_{T,K})$,
the path $Q_t:=Q \left(T, w_t, \epsilon \cdot \sin( t \cdot \pi)\right)$  
has an associated path $u_t$ of second Neummann
eigenfunctions that satisfy the conditions (2) through (5) of Proposition 
\ref{prop:blocking}.
\end{lem}

\begin{proof}
By the defintion of $\delta_{T,K}$, the quadrilateral 
$Q(T,w_t,\epsilon)$ satisfies (a), (b), and (c) of 
Lemma \ref{lem:perturb-quad}. Condition (a) implies that 
there exists a path $u_t$ of eigenfunctions of $Q_t$.
Condition (b) implies that $u_t$ satisfies condition (3) in 
Proposition \ref{prop:blocking}, and condition (c) implies 
that condition (2) is satisfied.

Let $e_t$ be the side so that $e_0$ is the
component of $e \setminus \{p\}$ that contains $p$.
It follows that conditions (4) and (5) of 
Proposition \ref{prop:blocking} are satisfied.
\end{proof}

\begin{thm}
\label{thm:unstable-break}
Suppose that each convex quadrilateral has no interior critical points.
Let $T \in \Ncal$ and let $e$ be the side of $T$ that contains 
the index $-1$ critical point. Then for each $\eta>0$ there exists 
$\epsilon \in (0, \eta)$ and $w$ in the interior of $e$ so that each
second Neumann eigenfunction $u$ of $Q(T,w, \epsilon)$ has an unstable
critical point.
\end{thm}

\begin{proof}
Lemma \ref{lem:blocking-fam} provides us with a 
family of quadrilaterals $Q_t$ and second Neumann eigenfunctions $u_t$
that satisfy conditions (2) through (5) of Proposition \ref{prop:blocking}.
If each second Neumann eigenfuction on a quadrilaterals were to have no
interior critical points, then each $u_t$ would also satisfy
condition (1). Therefore, Proposition \ref{prop:blocking} would
imply that for some $t$ the function $u_t$ has an index zero 
critical point.
\end{proof}

\section{Second Neumann eigenfunctions on convex polygons}
\label{sec:convex}

\begin{prop} 
\label{prop:rotate-nodal}
Suppose that $P$ is convex without right angles
and suppose that $w$ lies in the interior of $P$. An arc of $\Zcal(R_w u)$
ends at a vertex $v$ of $P$ if and only if $v$ is a local extremum 
of $u$.
\end{prop}

\begin{proof}
If $w$ lies in $P$, then it lies in the interior of the 
sector associated to $v$.
By assumption the angle at $v$ lies in either $(0, \pi/2)$ or 
 $(\pi/2, \pi)$. The claim then follows from combining 
Corollary \ref{coro-rotation-nodal-arc}, 
Corollary \ref{coro:c_1-zero-nonzero-index},
and Corollary \ref{coro:1-connected-index}.
\end{proof}

With additional hypotheses, we can expand the scope of
Proposition \ref{prop:degree-1-stable} to include degree 1 vertices
of $\Zcal(R_wu)$ that are vertices of $P$.

\begin{coro} 
\label{coro:rotational-degree-1-stable}
Let $u$ be a second Neumann eigenfunction of a convex polygon $P$
with no right angles. If $w$ lies in the interior of $P$ 
then each degree one vertex of $\Zcal(R_w u)$ is a 
nonzero index critical point.
\end{coro}

\begin{proof}
Each degree 1 vertex $p$ of $\Zcal(R_w u)$ lies in $\partial P$.
If $p$ lies in the interior of an edge, then 
Proposition \ref{prop:degree-1-stable} applies.
If $p$ is a vertex, then Proposition \ref{prop:rotate-nodal}
applies.
\end{proof}

\begin{prop}
\label{interior_crit}
Let $u$ be a second Neumann eigenfunction $u$ on a convex polygon $P$. 
If $u$ has a critical point $p$ that lies in the interior of $P$, 
then $u$ has at least four nonzero index critical points on the boundary. 
In particular, $u$ has at least five critical points. 
\end{prop}

\begin{proof}
Without loss of generality $p=0$. Since $p$ is a critical point of $u$
we have 
\[
u(z)~
=~
u(0)~
+~
a \cdot x^2 + b \cdot xy + c \cdot y^2~
+~
O(|z|^3) 
\]
for some constants $a$, $b$ and $c$. 
We have $R_p u=-y\partial_x + x \partial_y$ and hence 
\[ 
R_p u(z)~
=~
b \cdot (x^2 - y^2)~ +~ 2(c-a)\cdot xy~ -~ b \cdot y^2~ + O(|z|^3).
\]
In particular, $p=0$ is a nodal critical point of the 
Laplace eigenfunction $R_pu$. Thus, by the result of \cite{Cheng},   
the valence of $\Zcal(R_p u)$ at $p$ is at least four. 
By Proposition 6.2 in \cite{J-M}, the nodal set $\Zcal(R_p u)$
is a tree whose degree 1 vertices lie in the boundary of $P$.  
Thus $\Zcal(R_p u)$ has at least four degree 1 vertices, 
and each of these is a nonzero index critical point
by Corollary \ref{coro:rotational-degree-1-stable}.
\end{proof}

\begin{coro} 
If $u$ has has only three critical points, 
then each critical point lies on the boundary.
Moreover, one critical point is a global maximum,
one critical point is a global minimum, and the third
critical point has index zero. 
\end{coro}

\begin{proof}
By Proposition \ref{interior_crit}, each critical point lies
on the boundary.  Since $u$ is nonconstant, at least two of these 
critical points are global extrema.  
The index of each global extremum
is $+1$. 
Thus, if there are exactly three critical points, then
it follows from 
Proposition \ref{prop:index-thm} that two critical
points have index 1 and the third has index zero.
\end{proof}


\end{document}